\documentclass[a4paper,12pt]{amsart}

\usepackage{amscd}
\usepackage[all]{xy}
\usepackage[dvips]{graphicx}
\title[Symmetries of center singularities]{Symmetries of center singularities of plane vector fields}
\author{Sergiy Maksymenko}
\address{Topology dept., Institute of Mathematics of NAS of Ukraine, Tere\-shchenkivska st. 3, Kyiv, 01601 Ukraine}
\email{maks@imath.kiev.ua}
\makeatletter
\newcommand\testshape{family=\f@family; series=\f@series; shape=\f@shape.}
\def\myemphInternal#1{\if n\f@shape%
\begingroup\itshape #1\endgroup\/%
\else\begingroup\bfseries #1\endgroup%
\fi}
\def\myemph{\futurelet\testchar\MaybeOptArgmyemph}
\def\MaybeOptArgmyemph{\ifx[\testchar \let\next\OptArgmyemph
                 \else \let\next\NoOptArgmyemph \fi \next}
\def\OptArgmyemph[#1]#2{\index{#1}\myemphInternal{#2}}
\def\NoOptArgmyemph#1{\myemphInternal{#1}}
\makeatother

\theoremstyle{plain}
\newtheorem{theorem}[subsection]{Theorem}
\newtheorem{lemma}[subsection]{Lemma}
\newtheorem{proposition}[subsection]{Proposition}
\newtheorem{corollary}[subsection]{Corollary}
\newtheorem{claim}[subsection]{Claim}

\newtheorem{remark}[subsection]{Remark}
\newtheorem{example}[subsection]{Example}

\theoremstyle{definition}
\newtheorem{definition}[subsection]{Definition}

\newcommand\CCC{\mathbb{C}}
\newcommand\NNN{\mathbb{N}}
\newcommand\RRR{\mathbb{R}}
\newcommand\ZZZ{\mathbb{Z}}

\newcommand\id{\mathrm{id}}
\newcommand\im{\mathrm{im\,}}
\newcommand\Int{\mathrm{Int}}
\newcommand\Per{\mathrm{Per}}

\newcommand\eps{\varepsilon}

\newcommand\imSh[1]{Sh(#1)}
\newcommand\AFld{F}
\newcommand\AFlow{\mathbf{F}}
\newcommand\singA{\Sigma_{\AFld}}

\newcommand\ShA{\varphi}
\newcommand\imShA{\imSh{\AFld}}
\newcommand\kerA{\ker(\ShA)}
\newcommand\BFld{G}
\newcommand\BFlow{\mathbf{G}}

\newcommand\ShB{\psi}
\newcommand\imShB{\imSh{\BFld}}

\newcommand\afunc{\alpha}
\newcommand\bfunc{\beta}

\newcommand\func{f}
\newcommand\gfunc{g}

\newcommand\gdif{g}
\newcommand\hdif{h}

\newcommand\Mman{M}

\newcommand\Vman{V}

\newcommand\Aut{\mathrm{Aut}}

\newcommand\orig{O}
\newcommand\orb{o}

\begin{document}

\begin{abstract}
Let $D^2 \subset\mathbb{R}^2$ be a closed unit $2$-disk centered at the origin $O\in \mathbb{R}^2$, and $F$ be a smooth vector field such that $O$ is a unique singular point of $F$ and all other orbits of $F$ are simple closed curves wrapping once around $O$.
Thus topologically $O$ is a ``center'' singularity.
Let $\theta: D^2 \setminus\{O\}\to(0,+\infty)$ be the function associating to each $z\not=O$ its period with respect to $F$. 
This function can be discontinuous at $O$.

Let $\mathcal{D}^{+}(F)$ be the group of all diffeomorphisms of $D^2$ which preserve orientation and orbits of $F$.
Under assumption that $\theta$ smoothly extends to all of $D^2$ we prove that $\mathcal{D}^{+}(F)$ is homotopy equivalent to the circle.
\end{abstract}

\maketitle
\newcommand\flatsign[1]{\bar{#1}}
\newcommand\liftsign[1]{\widetilde{#1}}
\newcommand\flowsign[1]{\mathbb{#1}}

\newcommand\TC{TC}
\newcommand\PTC{PTC}
\newcommand\Stabf{\mathcal{S}(f)}

\newcommand\EA{\mathcal{E}(\AFld)}
\newcommand\EApl{\mathcal{E}^{+}(\AFld)}
\newcommand\EAd{\mathcal{E}^{+}(\AFld,\partial)}

\newcommand\DA{\mathcal{D}(\AFld)}
\newcommand\DApl{\mathcal{D}^{+}(\AFld)}
\newcommand\DAd{\mathcal{D}^{+}(\AFld,\partial)}

\newcommand\EB{\mathcal{E}(\BFld)}
\newcommand\EBpl{\mathcal{E}^{+}(\BFld)}

\newcommand\EAp[1]{\mathcal{E}^{+}(\AFld)_1}

\newcommand\Df{\mathcal{D}^{+}(\func)}
\newcommand\Dfd{\mathcal{D}^{+}(\func,\partial)}

\newcommand\GLR[1]{\mathrm{GL}(#1,\RRR)}
\newcommand\GLRpl[1]{\mathrm{GL}^{+}(#1,\RRR)}

\newcommand\jet{j^1} 
\newcommand\jetinv{\jet^{-1}}

\newcommand\disk{D^2}
\newcommand\disko{\disk\setminus\orig}
\newcommand\tdisk{\widetilde{\disk}}

\newcommand\canon[1]{\mathsf{c}_{#1}}
\newcommand\cls[1]{[#1]}

\newcommand\Cinf{C^{\infty}}
\newcommand\Ci[2]{\Cinf(#1,#2)}
\newcommand\Cr[1]{C^{#1}}

\newcommand\XX{\mathcal{X}}
\newcommand\YY{\mathcal{Y}}

\newcommand\Wr[1]{\mathsf{W}^{#1}}
\newcommand\Sr[1]{\mathsf{S}^{#1}}

\newcommand\Hman{\mathbb{H}}
\newcommand\IHman{\mathring{\Hman}}
\newcommand\dHman{\partial\Hman}

\newcommand\thdif{\liftsign{\hdif}}
\newcommand\tgdif{\liftsign{\gdif}}
\newcommand\rrho{r}

\newcommand\tafunc{\liftsign{\afunc}}
\newcommand\tbfunc{\liftsign{\bfunc}}
\newcommand\tmu{\liftsign{\mu}}
\newcommand\tsigma{\liftsign{\sigma}}
\newcommand\tzeta{\liftsign{\zeta}}
\newcommand\txi{\liftsign{\xi}}

\newcommand\dd[1]{\frac{\partial}{\partial #1}}
\newcommand\ddd[2]{\frac{\partial #1}{\partial #2}}

\newcommand\dAx{\flatsign{X}}
\newcommand\dAy{\flatsign{Y}}
\newcommand\dAp{\flatsign{\Phi}}
\newcommand\dAr{\flatsign{R}}

\newcommand\CiZHR{\Cinf_{\ZZZ}(\Hman,\RRR)}
\newcommand\EBZpl{\mathcal{E}_{\ZZZ}^{+}(\BFld)}
\newcommand\FlatZHR{\mathsf{Flat}_{\ZZZ}(\Hman,\dHman)}
\newcommand\FlatHR{\mathsf{Flat}(\Hman,\dHman)}

\newcommand\FlatOR{\mathsf{Flat}(\RRR^2,\orig)}
\newcommand\contW[2]{\mathsf{W}^{#1,#2}}
\newcommand\contS[2]{\mathsf{S}^{#1,#2}}

\newcommand\thp{\liftsign{\Phi}} 
\newcommand\thr{\liftsign{R}} 

\newcommand\hx{X} 
\newcommand\hy{Y} 

\newcommand\tfunc{\liftsign{\func}}
\newcommand\torb{\liftsign{\orb}}

\newcommand\MapR{\mathsf{Map}^{\infty}(\RRR^2,0)}
\newcommand\MapH{\mathsf{Map}^{\infty}(\Hman,\dHman)}
\newcommand\MapZH{\mathsf{Map}^{\infty}_{\ZZZ}(\Hman,\dHman)}

\newcommand\jo[1]{j^{1}#1(\orig)}
\newcommand\jr[1]{j^{1}#1}

\newcommand\ml{\mathsf{lift}}
\newcommand\ff{\mathsf{fl}}

\newcommand\tg{\mathrm{tg}}
\newcommand\arctg{\mathrm{arctg}}

\newcommand\BP{\mathbf{\Phi}}
\newcommand\BR{\mathbf{R}}

\newcommand\Nbh{\mathcal{N}}

\newcommand\tz{\liftsign{z}}

\newcommand\FuncVanD{\mathcal{A}}

\newcommand\RE{Re}

\newcommand\rank{\mathrm{rank}}
\newcommand\spectr{\mathrm{sp}}
\section{Introduction}
Let $\disk\subset\RRR^2$ be a unit $2$-disk centered at the origin $\orig\in\Int\disk$ and $\AFld$ be a $\Cinf$ vector field with the following properties:
\begin{enumerate}
\item[(T1)] $\AFld$ is tangent to $\partial\disk$;
\item[(T2)] $\orig$ is a unique singular point of $\AFld$;
\item[(T3)] all other orbits of $\AFld$ are closed.
\end{enumerate}
Then it is easy to find a \myemph{homeomorphism} 
\begin{equation}\label{equ:h_rounding_orbits}
\hdif:\disk\to\disk 
\end{equation}
such that for each orbit $\orb$ of $\AFld$ its image $\hdif(\orb)$ is the circle of some radius $c\in(0,1]$ centered at origin, see Figure~\ref{fig:center_point}.
Therefore we will call a vector field $\AFld$ on $\disk$ satisfying (T1)-(T3) a \myemph{\TC\ vector field} and its singular point $\orig$ will be called a \myemph{topological center}.

\begin{figure}[ht]
\includegraphics[height=2cm]{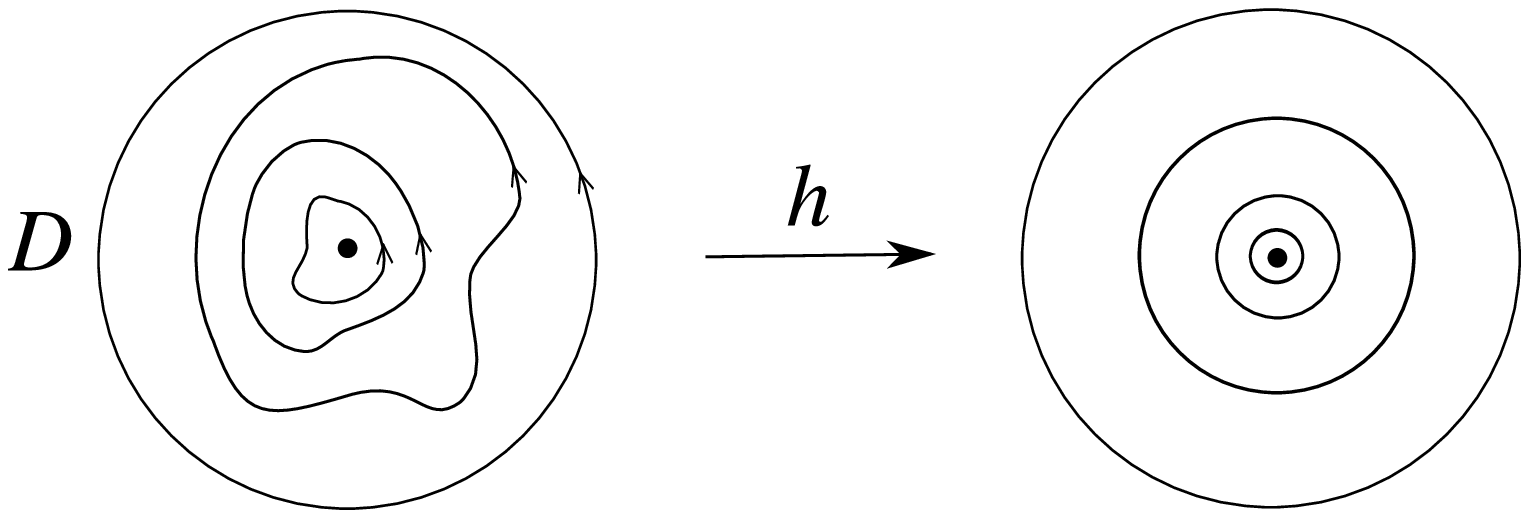}
\caption{}\label{fig:center_point}
\end{figure}

\TC\ vector fields often arise as Hamiltonian vector fields of local extremes of functions on surfaces.
For instance, let $\func:\disk\to[0,1]$ be a $\Cinf$ function such that $\func^{-1}(1)=\partial\disk$, 
$\func^{-1}(0)=\orig$, and $\orig$ is a unique critical point of $\func$ (being a global minimum of $\func$), see Figure~\ref{fig:func}:
\begin{figure}[ht]
\includegraphics[height=2cm]{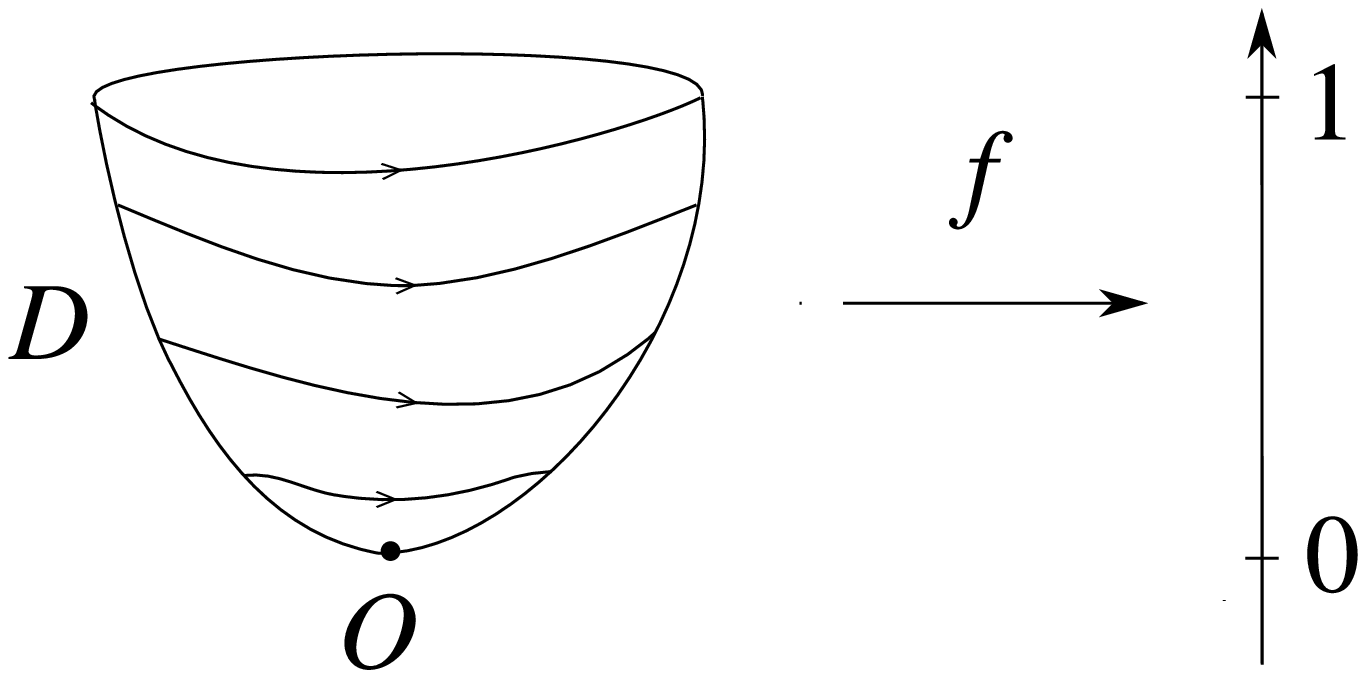}
\caption{}\label{fig:func}
\end{figure}
Then its Hamiltonian vector field $\AFld(x,y)=-\func'_{y}\dd{x}+\func'_{x}\dd{y}$ is \TC\ and $\func$ is its $\Cinf$ strong first integral in the following sense.

\begin{definition}{\rm c.f.~\cite{CGGL:JDE:99}}
A $\Cinf$ function $\func:\disk\to\RRR$ is a \myemph{strong first integral} for $\AFld$, if $\orig$ is a unique critical point of $\func$ and the Lie derivative $\AFld(\func)\equiv 0$, i.e. $\func$ is constant along orbits of $\AFld$.
\end{definition}

Since $\AFld$ is $\Cinf$, we can also assume that a homeomorphism $\hdif$ in~\eqref{equ:h_rounding_orbits} diffeomorphically maps $\disko$ onto itself, though it may loose differentiability at $\orig$.
In this case $\AFld$ has a \myemph{continuous} first integral on $\disk$ defined e.g. by $\func(z) = |\hdif(z)|^2$.
This function is $\Cinf$ on $\disko$ provided so is $\hdif$, but $\func$ is not necessarily smooth at $\orig$.

\medskip

Denote by $\DApl$ the group of $\Cinf$ orientation preserving diffeomorphisms $\hdif$ of $D^2$ such that $\hdif(\orb)=\orb$ for each orbit $\orb$ of $\AFld$.
Let also $\DAd$ be a subgroup of $\DApl$ consisting of diffeomorphisms fixed of $\partial\disk$.
We endow $\DApl$ and $\DAd$ with the weak $\Wr{\infty}$ Whitney topology.
The aim of the present paper is to describe the homotopy types of $\DApl$ and $\DAd$.

\medskip 

Let $\theta:\disko\to(0,+\infty)$ be the function associating to each $z\in\disko$ its period $\theta(z)$ with respect to $\AFld$.
We will call $\theta$ the \myemph{period function}.
Then it is easy to see that $\theta$ is $\Cinf$ on $\disko$, but it general it can not be even continuously extended to all of $\disk$.

\begin{example}\label{exmp:non-deg-hamvf}
Let $\AFld(x,y)=-y\dd{x}+x\dd{y}$.
Then $\theta \equiv 2\pi$, and therefore it is $\Cinf$ on all of $\disk$.
It follows from~\cite{Maks:Shifts, Maks:LocInv} that $\DApl$ is homotopy equivalent to $S^1$.
The generator of $\pi_1\DApl$ is given by the following isotopy 
$$
H:\DApl\times I\to\DApl,\qquad
H(\hdif,t)(z)=e^{2\pi i t}\hdif(z).
$$
\end{example}

\begin{example}\label{exmp:deg-hamvf}
Let $Q_1,\ldots,Q_n:\RRR^2\to\RRR$ be definite (that is irreducible over $\RRR$) quadratic forms such that $Q_i/Q_j\not=\mathrm{const}$ for $i\not=j$,
$$\func=Q_1\cdots Q_n,$$ and $\AFld(x,y)=-\func'_{y}\dd{x}+\func'_{x}\dd{y}$ be the Hamiltonian vector field of $\func$.

If $n\geq2$, then $\lim\limits_{z\to\orig}\theta(z)=+\infty$, whence $\theta$ can not be continuously extended to $\disk$.
Then, \cite{Maks:part-pres-diff},
$\DApl$ is \myemph{path-connected with respect to $\Wr{0}$ topology}.
On the other hand in all others $\Wr{r}$ topologies ($r\geq1$) the group $\pi_0\DApl$ is (the same for all $r\geq1$) \myemph{non-trivial finite cyclic group of even order}.
Moreover, each path component of $\DApl$ is contractible with respect to $\Wr{\infty}$ topology.
\end{example}

It turns out that these examples describe all the possibilities for $\theta$.
Actually the following theorem holds true:

\begin{theorem}\label{th:charact_period_shift_maps}
Let $\AFld=\AFld_1\dd{x} + \AFld_2\dd{y}$ be a \TC\ vector field on $\disk$ and let $\theta:\disko\to(0,+\infty)$ be its period function.
Then the following conditions are equivalent:
\begin{enumerate}
 \item[\rm(a)] $\theta$ smoothly extends to all of $\disk$;
 \item[\rm(b)] the eigen values of the matrix 
$$ 
\nabla\AFld = \left(\begin{array}{cc}
\ddd{\AFld_1}{x} &  \ddd{\AFld_1}{y} \\ [2mm]
\ddd{\AFld_2}{x} &  \ddd{\AFld_2}{y}
\end{array} \right)
$$ 
at $\orig$ are non-zero purely imaginary;
\item[\rm(c)] there exists a $\Cinf$ function $\bfunc:\disk\to\RRR$ and a diffeomorphism $\gdif:(\disk,\orig)\to(\disk,\orig)$ such that $\bfunc(\orig)\not=0$ and
$$
\gdif_{*}\AFld = \bfunc(x^2+y^2) \left(-y \frac{\partial}{\partial x} + x \frac{\partial}{\partial y}\right) + \dAx\dd{x}+ \dAy\dd{y},
$$
where $\dAx,\dAy\in\FlatOR$.
\end{enumerate}

If either of these conditions fails, then $\lim\limits_{z\to\orig}\theta(z)=+\infty$.
\end{theorem}
The implication (a)$\Rightarrow$(b) follows from~\cite{Maks:ReparamShMap}, and (b)$\Rightarrow$(c) from Takens~\cite{Takens:AIF:1973}.

A \TC\ vector field satisfying one of the conditions (a)-(c) of Theorem~\ref{th:charact_period_shift_maps} will be called \PTC. This notation reflects \myemph{periodicity} of shift map, see \S\ref{sect:shift-map}.

The main result of this paper is contained in the following theorem:
\begin{theorem}\label{th:DApl_DAd_hom_types}
 If $\AFld$ is a \PTC\ vector field on $\disk$, then $\DApl$ is homotopy equivalent to the circle, and $\DAd$ is contractible with respect to $\Wr{\infty}$-topologies.
\end{theorem}
Thus for \PTC\ vector fields description of homotopy types of $\DApl$ and $\DAd$ is the same as in Example~\ref{exmp:non-deg-hamvf}.
This result is a particular case of Theorem~\ref{th:shmap_periodic}.

\begin{remark}\rm
Suppose $\AFld$ has a $\Cinf$ strong first integral $\func:\disk\to[0,1]$ being a surjective function.
Then for each $c\in[0,1]$ its inverse image $\func^{-1}(c)$ is the orbit of $\AFld$.
It follows that each $\hdif\in\DApl$ preserves $\func$, i.e. $\func\circ\hdif=\func$.
Thus we can regard $\DApl$ as the stabilizer $\Stabf$ of $\func$ with respect to the \myemph{right} action of the group of orbit preserving diffeomorphisms of $\disk$ on $\Ci{\disk}{\RRR}$.

There is a vast of results concerning diffeomorphisms preserving functions.
Most of them deal with actions of compact Lie groups, see e.g.~\cite{Schwarz:T:75, Huffman:CJM:80}.
From this point of view Theorem~\ref{th:DApl_DAd_hom_types} describes the (infinite-dimensional) group of \myemph{all} orientation preserving symmetries of $\func$ but for a very specific case.
This theorem will be used in another papers for the description of the homotopy types of stabilizers and orbits of smooth functions on surfaces, which will extend results of~\cite{Maks:AGAG:2006} on Morse functions.
\end{remark}

Notice that due to Takens~\cite{Takens:AIF:1973} (see (c) of Theorem~\ref{th:charact_period_shift_maps}), a \PTC\ vector field $\AFld$ is a ``flat perturbation'' of the vector field of  Example~\ref{exmp:non-deg-hamvf}.
We prove that $\AFld$ is parameter rigid (see Claim~\ref{clm:F_is_param_rig}) and using this fact give another proof that any $\Cinf$ strong first integral $\func:\disk\to\RRR$ of $\AFld$ is a ``flat perturbation'' of a smooth function depending on $x^2+y^2$, see~\cite{ArnoldYllyashenko}:
\begin{theorem}\label{th:func_flat_pert}
Let $\AFld$ be a \PTC\ vector field having a $\Cinf$ strong first integral $\func:\disk\to\RRR$.
Then there exist $\Cinf$ functions $\gfunc:\RRR\to\RRR$ and $\mu:\disk\to\RRR$ such that $\mu$ is flat at $\orig$ such that 
$$
\func(x,y) = \gfunc(x^2+y^2) + \mu(x,y).
$$
\end{theorem}
Existence of first integrals for such \PTC\ vector fields and the problem of recognizing of \TC\ integrals with non-degenerate linear part are studied in~\cite{Poincare, Lyapunov, Moussu:ETDS:82, Sibirsky1, Sibirsky2}.
See also~\cite[Chapter 5, \S4]{ArnoldYllyashenko} for a review of the problem and references.

\medskip 

The case of non \PTC\ vector field, i.e. when $\lim\limits_{z\to\orig}\theta(z)=+\infty$, will be considered in another paper, where it will be shown that under additional assumptions on $\AFld$ the description of $\DApl$ and $\DAd$ is similar to Example~\ref{exmp:deg-hamvf}.

\subsection{Structure of the paper.}
In \S\ref{sect:shift-map} we recall the notion of the shift map of a vector field and formulate Theorem~\ref{th:shmap_periodic}.

In \S\ref{sect:linearization}--\S\ref{sect:Jacobi_matr} linear parts of \TC\ vector fields and diffeomorphisms preserving their orbits are studied.

The idea of proofs of Theorems~\ref{th:charact_period_shift_maps} and~\ref{th:shmap_periodic} is to reduce $\AFld$ to a certain normal form and then ``blow up'' the singularity at $\orig$ by using polar coordinates.
Therefore in \S\ref{sect:polar_coordinates} we give conditions when a smooth functions, self-maps, and vector fields on $\disk$ yield the corresponding smooth objects in polar coordinates and vice versa.

In \S\ref{sect:Takens_nf} we recall the result of Takens~\cite{Takens:AIF:1973} about normal forms of vector fields on $\RRR^2$ with ``rotation as $1$-jet''.
We also study the formulas for the flows of these vector fields with respect to polar coordinates.

The rest of the paper is devoted to the proofs of Theorems~\ref{th:charact_period_shift_maps}, \ref{th:shmap_periodic}, and \ref{th:func_flat_pert}.

\subsection{Notations}
Let $f=(f_1,\ldots,f_m):\RRR^n\to\RRR^m$ be a $\Cinf$ map, $K\subset\RRR^n$ a compact subset, and $k\in\{0\}\cup\NNN$.
Then the \myemph{$k$-norm of $f$ on $K$} is defined by
$$
\|f\|^k_{K} = 
\sup_{x\in K}
\ \sum_{j=1}^{m} \
\sum_{|\alpha|\leq k} \frac{\partial^{|\alpha|} f_j}{\partial x^{\alpha}},
$$ 
where $\alpha=(\alpha_1,\ldots,\alpha_n)$, $\alpha_i\in\{0\}\cup\NNN$, and $|\alpha|=\sum_{i=1}^{n}\alpha_i$.
For a fixed $k$ the norm $\|\cdot\|^k_{K}$ generate the weak $\Wr{k}$ topology on $\Ci{\RRR^n}{\RRR^m}$.

More generally, let $A$ and $B$ be smooth manifolds.
Then for every $r=0,1,\ldots,\infty$ we can define the weak $\Wr{r}$- and the strong $\Sr{r}$ topologies on $\Ci{A}{B}$, see e.g.~\cite{Hirsch:DiffTop}.
We will assume that the reader is familiar with them.

Let $f\in\Ci{A}{B}$, $a\in A$, and $k\in\NNN\cup\{\infty\}$.
Then by $j^k f(z)$ we will denote the $k$-jet of $f$ at $z$.

A subset $\XX\subset\Ci{A}{B}$ will be called $\Wr{r}$-open ($\Wr{r}$-closed, etc.)\! if it is open (closed) with respect to the $\Wr{r}$ topology on $\Ci{A}{B}$.
Let also $C$ and $D$ be some other smooth manifolds, $\YY\subset \Ci{C}{D}$ be a subset, and $u:\XX\to\YY$ be a map.
Then $u$ will be called  \myemph{$\contW{s}{r}$-continuous} (\myemph{$\contW{s}{r}$-open} etc.) if it is continuous (open) from $\Wr{s}$-topology of $\XX$ to $\Wr{r}$-topology of $\YY$, $(r,s=0,1,\ldots,\infty)$.

Similarly, we can define \myemph{$\contS{s}{r}$-continuity}, \emph{$\contS{s}{r}$-openness} of maps, and \myemph{$\Sr{r}$-openness} of subsets.

\begin{definition}\label{defn:pres_smoothness}
We will say that $u:\XX\to\YY$ \myemph{preserves smoothness} if for any $C^{\infty}$ map $H:A\times \RRR^n\to B$ such that $H_t=H(\cdot,t)\in \XX$ for all $t\in\RRR^n$ the following mapping
$$
u(H): C\times \RRR^n\to D, 
\qquad 
u(H)(c,t) = u(H_t)(c)
$$
is $C^{\infty}$ as well.
\end{definition}
For instance, if $f:A\to B$ and $g:C\to D$ are $\Cinf$ maps, then the mapping 
$$
u:\Ci{B}{C} \to \Ci{A}{D},
\qquad u(\alpha) = g\circ \alpha \circ f,
$$
for $\alpha\in\Ci{A}{B}$ preserves smoothness.

\section{Shift map}\label{sect:shift-map}
In this section we formulate Theorem~\ref{th:shmap_periodic} containing Theorem~\ref{th:DApl_DAd_hom_types}.

Let $\AFld$ be a $\Cinf$ vector field of $\disk$.
Denote by $\EA$ the subset of $\Ci{\disk}{\disk}$ consisting of mappings $\hdif:\disk\to\disk$ having the following properties:
\begin{itemize}
 \item 
$\hdif(\omega)=\omega$ for every orbit $\omega$ of $\AFld$.
In particular, $\hdif(\orig)=\orig$.

 \item
$\hdif$ is local diffeomorphism at $\orig$ preserving orientation, that is the tangent map $T_{\orig}\hdif:T_{\orig}\disk\to T_{\orig}\disk$ is a non-degenerate linear map, and the Jacobian $|J(\hdif,\orig)|\not=0$, \cite[Cor.~21]{Maks:Shifts}.
\end{itemize}
Let also $\EAd \subset \EApl$ be the subset consisting of all maps $\hdif$ fixed on $\partial\disk$, i.e. $\hdif(x)=x$ for all $x\in\partial\disk$.

Notice that $\partial\disk$ is the orbit of $\AFld$.
Since, in addition, $\disk$ is compact, it follows that $\AFld$ generates a global flow  $\AFlow:\disk\times\RRR\to\disk$ on $\disk$.
Then we can define the following map
$$\ShA:\Ci{\disk}{\RRR}\to \Ci{\disk}{\disk},
\qquad
\ShA(\afunc)(x)=\AFlow(x,\afunc(x))
$$
where $\afunc\in\Ci{\disk}{\RRR}$ and $x\in\disk$.
We will call $\ShA$ the \myemph{shift map} along orbits of $\AFld$, see~\cite{Maks:Shifts}.
Denote by $\imShA$ the image of $\ShA$ in $\Ci{\disk}{\disk}$.

Let $\Vman\subset\disk$ be a subset, $\afunc:\Vman\to\RRR$ a function, and $\hdif:\disk\to\disk$ a map.
We will say that $\afunc$ is a \myemph{shift function} for $\hdif$ provided that $\hdif(z)=\AFlow(z,\afunc(z))$ for all $z\in\Vman$.
In particular, $\theta$ is the shift function for the identity map $\id_{\disk}$ on $\disko$.

It is easy to see,~\cite[Cor.~21]{Maks:Shifts}, that $\imShA\subset\EApl$.
Endow $\imShA$, $\EAd$, $\EApl$, and $\EA$ with the corresponding weak $\Wr{\infty}$ Whitney topologies.

The following set $\kerA=\ShA^{-1}(\id_{\disk})$ will be called the \myemph{kernel} of shift map.
It consist of $\Cinf$ functions $\mu:\disk\to\RRR$ such that $\AFlow(z,\mu(z))=z$ for all $z\in\disk$.
It is shown in~\cite[Cor.~6]{Maks:Shifts} that $$ \ShA^{-1}\ShA(\afunc) = \afunc + \kerA$$
for each $\afunc\in\Ci{\disk}{\RRR}$.

Since the set $\{\orig\}$ of singular points of $\AFld$ is nowhere dense in $\disk$, it follows from~\cite[Th.~12 \& Pr.~13]{Maks:Shifts} that the shift map $\ShA$ of $\AFld$ is locally injective even with respect $\Wr{0}$-topology of $\Ci{\disk}{\RRR}$ and there are the following two possibilities for $\kerA$:

{\bf Periodic case}.
$\kerA=\{n\mu\}_{n\in\ZZZ}$ for some $\Cinf$ strictly positive function $\mu:\disk\to(0,+\infty)$.
This function  will be called the \myemph{period} function for $\ShA$.

{\bf Non-periodic case}. $\kerA=0$, so $\ShA$ is injective map.
This case holds when $\AFld$ has at least one non-closed orbit, or the linear part, i.e. $1$-jet $j^{1}\AFld(z)$ at some singular point $z\in\disk$ of $\AFld$ vanish~\cite[Pr.~10]{Maks:Shifts}.

\begin{remark}\label{rem:eigen_val_vanish}
In fact, it can be proved similarly to~\cite[Pr.~10]{Maks:Shifts} that $\kerA=0$ provided \myemph{only the eigen values of $j^1\AFld(z)$ vanish}.

Moreover, in this case for any sequence of \myemph{periodic points} $\{z_i\}_{i\in\NNN}$ converging to $z$ (if such a sequence exists) their periods tend to infinity.
This remark will be used in the proof of Theorem~\ref{th:charact_period_shift_maps}.
\end{remark}

\medskip 

Evidently, if $\AFld$ is a \PTC\ vector field, that is its period function $\theta:\disko\to(0,+\infty)$ smoothly extends to all of $\disk$, then $\ShA$ is periodic, and $\theta=\mu$.

\begin{theorem}\label{th:shmap_periodic}
Let $\AFld$ be a \PTC\ vector field.
Then
\begin{enumerate}
\item
$\imShA=\EApl$  and the map $\ShA:\Ci{\disk}{\RRR}\to\EApl$ is an infinite cyclic covering map;
\item
the inclusions $\DApl\subset\EApl$ and $\DAd\subset\EAd$ are homotopy equivalences;
\item 
$\DApl$ and $\EApl$ are homotopy equivalent to the circle;
\item 
$\DAd$ and $\EAd$ are contractible.
\end{enumerate}
\end{theorem}

The proof will be given in \S\ref{sect:th:shmap_periodic}.

\subsection{Shift map of non-singular vector fields on $\RRR^2$}
We will also use the following statement.
\begin{lemma}\label{lm:shift-maps-without-sing-R2}
Let $\Mman$ be either the plane $\RRR^2$ or the half-plane $\Hman$ and $\BFld$ be a vector field on $\Mman$ tangent to $\partial\Mman$ (in the case $\Mman=\Hman$) and having no singular points.
Suppose that $\BFld$ generates a flow $\BFlow:\Mman\times\RRR\to\Mman$ and let $\ShB:\Ci{\Mman}{\RRR}\to\Ci{\Mman}{\Mman}$ be the shift map of $\Mman$.
Then its image $\imShB$ coincides with $\EB$, and the map $\ShB:\Ci{\Mman}{\RRR}\to\EB$ is a $\contS{r}{r}$-homeomorphism for all $r\geq0$.

Both maps maps $\ShB$ and $\ShB^{-1}$ preserve smoothness.
\end{lemma}
\begin{proof}
Let $\hdif\in\EB$.
Then for each $x\in\Mman$ its image $\hdif(x)$ belongs to the orbit of $x$, whence there exists a unique number $\sigma_{\hdif}(x)$ such that $\hdif(x)=\BFlow(x,\sigma_{\hdif}(x))$.
The obtained shift function $\sigma_{\hdif}:\Mman\to\RRR$ for $\hdif$ is $\Cinf$ on all of $\Mman$ and the correspondence $\hdif\mapsto\sigma_{\hdif}$ is the inverse map $\ShB^{-1}$ of $\ShB$.

Since $\BFld$ has no singular points, and $\Int\Mman$ is homeomorphic with $\RRR^2$, it easily follows from the Poincar\'e-Bendixson theorem, \cite{Palis_deMelu}, that $\BFld$ has no closed orbits and each non-closed orbit of $\BFld$ is non-recurrent.
Then it follows from~\cite{Maks:LocInv} that $\ShB^{-1}$ is $\contS{r}{r}$-continuous for each $r\geq0$ and preserve smoothness.
\end{proof}

\section{Linearization of vector fields}\label{sect:linearization}
Let 
\begin{equation}\label{equ:VF}
\AFld(x)=\AFld_1(x)\frac{\partial}{\partial x_1} + \cdots + \AFld_n(x)\frac{\partial}{\partial x_n} 
\end{equation}
be a smooth vector field on $\RRR^n$ such that $\AFld(\orig)=0$.
Then 
$$
\AFld_i(x) = a_{i1}x_1 + \cdots + a_{in}x_n + o(\|x\|^2), \qquad i=1,\ldots,n,
$$
for some $a_{ij}\in\RRR$.
Regarding $\AFld$ as a map $\AFld=(\AFld_1,\ldots,\AFld_n):\RRR^n\to\RRR^n$, we can write the $1$-jet $\jet\AFld(\orig)$ of $\AFld$ at $\orig$ as 
$
\jet(\AFld)(x) = Ax, 
$ 
where 
$$A=\left(
\begin{matrix}
a_{11} & \cdots & a_{1n} \\
\cdots & \cdots & \cdots \\
a_{n1} & \cdots & a_{nn}
\end{matrix}
\right),
\qquad 
x = 
\left(
\begin{matrix}
x_{1} \\ \cdots \\ x_{n}
\end{matrix}
\right).
$$
The matrix $A$ as well as the corresponding linear map $x\mapsto Ax$ will be denoted by $\nabla\AFld$ and called \myemph{linear part} or \myemph{linearization} of $\AFld$ at $\orig$.

Let $\hdif:(\RRR^n,\orig)\to(\RRR^n,\orig)$ be a germ of a diffeomorphism at $\orig$ and $H=J(\hdif,\orig)$ be its Jacobi matrix of $\hdif$ at $\orig$.
Then $\hdif$ induces the following germ of a vector field 
$$\hdif_{*}\AFld = T\hdif \circ \AFld \circ \hdif^{-1}$$ at $\orig$ called the \myemph{pushforward} of $\AFld$ by $\hdif$.

It easily follows that 
\begin{equation}\label{equ:nabla_for_pushforward}
\nabla \hdif_{*}\AFld = H \cdot \nabla\AFld \cdot H^{-1}.
\end{equation}

\begin{lemma}\label{lm:1jet_TCVF}
Let $\AFld$ be a \TC\ vector field on $\disk$.
Then the eigen values of its linearization $\nabla\AFld$ at $\orig$ are purely imaginary, i.e. $\lambda_{1,2}=\pm ib$ for some $b\in\RRR$.
Hence by change coordinates $\nabla\AFld$ can be reduced to one of the following matrices:
$$
1.~\left(\begin{smallmatrix}
0 & b \\ -b & 0 
\end{smallmatrix}\right),
\qquad
2.~\left(\begin{smallmatrix}
0 & b \\ 0 & 0 
\end{smallmatrix}\right),
\qquad
3.~\left(\begin{smallmatrix}
0 & 0 \\ 0 & 0 
\end{smallmatrix}\right),
$$
for some $b>0$.
Each of these matrices is realizable.
\end{lemma}
\begin{proof}
Suppose that $\RE\lambda_1\not=0$.
Then there exists an orbit $\orb$ of $\AFld$ for which $\orig$ is a limit point.

Indeed, if $\RE\lambda_2\not=0$, then existence of $\orb$ follows from Hadamard-Perron's theorem~\cite{Hadamard:BMSF:1901, Perron:MZ:1929, Perron:JRAM:1929, Perron:MZ:1930}, which was reproved and extended by many authors, see~\cite[page 2]{HirschPughShub:LNM} for discussions and references.
Otherwise $\RE\lambda_2=0$, and such an orbit $\orb$ exists by \myemph{center manifold theorem}, e.g.~\cite{Lykova:UMZ:59, Kelley:JDE:67, Carr:CMTh}.

This gives a contradiction with the assumption that all orbits of $\AFld$ are closed.

It remains to present examples of \TC\ vector field with linear parts of types 1-3.
Let $p,q\in\NNN$, $\func(x,y)=\frac{b}{2} (x^{2p}+ y^{2q})$, and 
$$\AFld(x,y) = -bq  y^{2q-1}\dd{x} + bp x^{2p-1}\dd{y}$$ be the Hamiltonian vector field of $\func$.
Since $\orig$ is an isolated local extreme for $\func$, we see obtain that $\AFld$ is a \PTC\ vector field.
Then 
$$
\begin{array}{ll}
\nabla\AFld=
\left(\begin{smallmatrix}
0 & b \\ -b & 0 
\end{smallmatrix}\right), & p=q=1, \\[3mm]
\nabla\AFld=
\left(\begin{smallmatrix}
0 & b \\ 0 & 0 
\end{smallmatrix}\right), & p=1, q=2, \\[3mm]
\nabla\AFld=
\left(\begin{smallmatrix}
0 & 0 \\ 0 & 0 
\end{smallmatrix}\right), & p,q\geq2.
\end{array}
$$
Lemma is completed.
\end{proof}

\begin{lemma}\label{lm:eigen_val_H}
Let $\hdif:(\disk,\orig)\to(\disk,\orig)$ be germ of orbit preserving diffeomorphisms of a \TC\ vector field $\AFld$, i.e. $\hdif(\orb)=\orb$ for each sufficiently small orbit of $\AFld$, $H=J(\hdif,\orig)$ be the Jacobi matrix of $\hdif$ at $\orig$, and $\mu_1,\mu_2$ be the eigen values of $H$.
Then $|\mu_1|=|\mu_2|=1$.
\end{lemma}
\begin{proof}
If $|\mu_1|\not=1$, then by center manifold theorem, e.g.~\cite{Carr:CMTh}, there exists a point $z\in\disko$ such that 
\begin{equation}\label{equ:orig_in_orbit_hz}
\orig \in \overline{ \mathop{\cup}\limits_{n=-\infty}^{+\infty}\hdif^n(z) }.
\end{equation}
Let $\orb$ be the orbit of $z$.
By assumption $\hdif(\orb)=\orb$, whence~\eqref{equ:orig_in_orbit_hz} implies that $z=\orig$, which contradicts to the assumption.
\end{proof}

\section{``Collinear'' linear maps}
The aim of this section is to establish the following statement.
\begin{proposition}\label{pr:collin-conjug-lin-maps}
Let $V$ be a linear space over a field $\mathbb{F}$, $\dim V=n\geq2$, and $A,B:V\to V$ be two linear maps such that 
\begin{enumerate}
 \item[\rm(i)]
$B = H A H^{-1}$ for some linear isomorphism $H:V\to V$;
 \item[\rm(ii)] 
for each $x\in V$ its images $A(x)$ and $B(x)$ are collinear, i.e. there exists $\mu_x\in\mathbb{F}$ depending on $x$ such that $B(x) = \mu_x A(x)$.
\end{enumerate}
Denote by $\spectr A$ the spectrum of $A$.
Then the following statements hold true:
\begin{enumerate}
 \item[\rm(A1)]
If $\rank A\geq2$, then $B= \tau A$ for some $\tau\in\mathbb{F}$.
In particular, $$A H \;=\; \tau H A.$$

\item[\rm(A2)]
Suppose $\rank A=1$ and $\spectr A=\{\lambda,0\}$ for some $\lambda\not=0$. Then there exists a basis in $V$ in which 
\begin{equation}\label{equ:A_nonzero_spectrum_B}
A= \left(\begin{smallmatrix}
\lambda &  0 & \cdots & 0\\
0   &  0  & \cdots  & 0\\
\cdots & \cdots &  \cdots & \cdots  \\
0   &  0  & \cdots  & 0
\end{smallmatrix}\right),
\qquad 
B = \left(\begin{smallmatrix}
\lambda &  b_2   & \cdots & b_n \\
    0   &  0     & \cdots & 0 \\
\cdots  & \cdots & \cdots & \cdots \\
0   &  0  & \cdots  & 0
\end{smallmatrix}\right),
\end{equation}
for some $b_2,\ldots,b_n\in\mathbb{F}$.
Denote 
\begin{equation}\label{equ:A_nonzero_spectrum_G1}
G_1= \left(\begin{smallmatrix}
1 & b_2/\lambda & b_3/\lambda & \cdots & b_n/\lambda \\
0 & 1/\lambda   & 0                    & \cdots & 0  \\
0 & 0           & 1/\lambda            & \cdots & 0  \\
\cdots & \cdots &  \cdots  & \cdots & \cdots \\
0 &       0     & 0                             & \cdots &  1/\lambda
\end{smallmatrix}\right).
\end{equation}
Then $B = G_1 B = A G_1$, whence 
$$A \cdot (H G_1) \;=\; (H G_1) \cdot A,$$
i.e. $A$ commutes with $HG_1$.

\item[\rm(A3)]
Suppose $\rank A=1$ and $\spectr A=\{0\}$.
Then there exists a basis in $V$ in which 
\begin{equation}\label{equ:A_zero_spectrum_B}
A= \left(\begin{smallmatrix}
0   &  1 & \cdots & 0\\
0   &  0  & \cdots  & 0\\
\cdots & \cdots &  \cdots & \cdots  \\
0   &  0  & \cdots  & 0
\end{smallmatrix}\right),
\qquad 
B = \left(\begin{smallmatrix}
0   &   q  & b_3  & \cdots & b_n \\
0   &   0  & 0 & \cdots & 0 \\
\cdots     & \cdots & \cdots & \cdots & \cdots \\
0   &  0   & 0 & \cdots  & 0
\end{smallmatrix}\right),
\end{equation}
for some $q,b_3,\ldots,b_n\in\mathbb{F}$ such that $q\not=0$.
Denote 
\begin{equation}\label{equ:A_zero_spectrum_G2}
G_2= \left(\begin{smallmatrix}
1 & 0     & 0 & \cdots & 0 \\
0 & q & b_3 & \cdots & b_n  \\
0 & 0           & 1/q     & \cdots & 0  \\
\cdots & \cdots &  \cdots  & \cdots & \cdots \\
0 &       0     & 0        & \cdots &  1/q
\end{smallmatrix}\right),
\end{equation}
Then $B = G_2 B = A G_2$, whence 
$$A \cdot (H G_2) \;=\; (H G_2) \cdot A,$$
i.e. $A$ commutes with $HG_2$.
\end{enumerate}
\end{proposition}

We need two lemmas.
Let $V,W$ be two linear spaces over a field $\mathbb{F}$.
For $w\in W$ denote by $$\langle w \rangle:=\{tw \ : \ t\in \mathbb{F}\}$$ the one-dimensional subspace in $W$ spanned by $w$.

\begin{lemma}\label{lm:ex_x_Ax_Bx_n0}
For any two non-zero linear operators $A,B: V\to W$ there exists $x\in V$ such that $A(x)\not=0$ and $B(x)\not=0$.
\end{lemma}
\begin{proof}
Suppose that for each $x\in V$ at least one of vectors $Ax$ or $Bx$ is zero.
Then $\ker A + \ker B =V$.
Since $A,B\not=0$, there exists $x\in \ker B\setminus \ker A$ and $y\in \ker A\setminus \ker B$.
In particular, $x,y$ are linearly independent, 
\begin{equation}\label{equ:AxBxAyBy}
A(x)\not=0,  \qquad B(y)\not=0, \qquad A(y)=B(x)=0.
\end{equation}
But $A(x+y) = A(x)$, $B(x+y) = B(y)$ and by assumption at least one of these vectors must be zero, which contradicts to~\eqref{equ:AxBxAyBy}.
\end{proof}

\begin{lemma}\label{lm:collin_AB_V_W}
Let $A,B:V \to W$ be two linear operators such that for each $x\in V$ the vectors $A(x)$ and $B(x)$ are collinear (possibly zero).
\begin{enumerate}
 \item[\rm(a)]
If $B(x)=\tau A(x) \not=0$ for some $x\in V$ and $\tau\in\mathbb{F}$, then
$$
A(\ker B)\subset \langle A(x) \rangle,
\qquad
B(\ker A)\subset \langle B(x) \rangle.
$$
\item[\rm(b)]
If $B(x)=\tau A(x) \not=0$, $B(y)=\nu A(y) \not=0$ for some $x,y\in V$ and $\tau,\nu\in\mathbb{F}$ such that $A(x)$ and $A(y)$ are linearly independent in $W$, then $\tau=\nu$.
Moreover, in this case $B=\tau A$.
\end{enumerate}
\end{lemma}
\begin{proof}
(a) Evidently, it suffices to show that \myemph{if $B(z)=0$ for some $z\in V$, then $A(z)=\kappa A(x)$ for some $\kappa\in\mathbb{F}$.}

By assumption $B(x+z)=\alpha A(x+z)$ for some $\alpha\in\mathbb{F}$.
Then 
$$
A(x) + A(z) = A(x+z)= \alpha B(x+z)= \alpha B(x) + \alpha B(z) =\alpha\kappa A(x),
$$
whence $A(z) = (\alpha \kappa -1) A(x)$.

(b)
By assumption $B(x+y)=\beta A(x+y)$ for some $\beta\in\mathbb{F}$.
Then 
$$
A(x) + A(y) = A(x+y)= \beta B(x+y)= \beta B(x) + \beta B(y) =\beta \tau A(x) + \beta \nu A(y),
$$
Since $A(x)$ and $A(y)$ are linearly independent, we obtain that $\beta\not=0$, and $\beta\tau=\beta\nu=1$, whence $\tau=\nu$.

It remains to show that $B(z)=\tau A(z)$ for any other $z\in V$.

If $B(z)=0$, then we get from (a) that $A(z) \in \langle A(x)\rangle \cap  \langle A(y)\rangle = \{0\}$.
By symmetry we obtain that $A(z)=0$ if and only if $B(z)=0$.
In particular, $B(z)=\tau A(z)=0$ for such $z$.

Suppose $A(z)\not=0$.
Then $B(z)\not=0$ as well, and either $A(x),A(z)$ or $A(y),A(z)$ are linearly independent.
Then as just proved $B(z)=\tau A(z)$.
\end{proof}

\subsection*{Proof of Proposition~\ref{pr:collin-conjug-lin-maps}}
(A1)
By Lemma~\ref{lm:ex_x_Ax_Bx_n0} there exists $x\in V$ such that $A(x)=\tau B(x)\not=0$ for some $\tau\in\mathbb{F}$.
Since $\rank A\geq2$, there exist $y\in V$ such that $A(x)$ and $A(y)$ are linearly independent.
Then by (b) of Lemma~\ref{lm:collin_AB_V_W} $A=\tau B$.

(A2), (A3)
It is easy to verify that if $A$ and $B$ are given either  by~\eqref{equ:A_nonzero_spectrum_B} or by ~\eqref{equ:A_zero_spectrum_B}.
Then $B=G_i B = AG_i$, where $G_i$ is given by the corresponding formula~\eqref{equ:A_nonzero_spectrum_G1} or~\eqref{equ:A_zero_spectrum_G2}.
Hence $B=HAH^{-1}=G_i^{-1} A G_i$ and therefore $(G_i H) A = A (G_i H)$.

We have to establish existence of representations~\eqref{equ:A_nonzero_spectrum_B} and~\eqref{equ:A_zero_spectrum_B}.

Suppose $\rank A=1$.
Then by Jordan normal form theorem that there exists a basis $\langle e_1,\ldots,e_n\rangle$ in $V$ in which $A$ is given by the corresponding matrix~\eqref{equ:A_nonzero_spectrum_B} or~\eqref{equ:A_zero_spectrum_B}.
Thus in both cases the image of $A$ is spanned by vector $e_1$.

Notice that by (A1) $\rank B=1$ and by (i) $\spectr  B=\spectr A$.
Then it follows from (ii) that the image of $B$ is also spanned by $e_1$.
Hence $B$ is given by the corresponding matrix~\eqref{equ:A_nonzero_spectrum_B} or~\eqref{equ:A_zero_spectrum_B}.

It remains to show that \myemph{we may assume in~\eqref{equ:A_zero_spectrum_B}  that $q\not=0$}.

Indeed, by Lemma~\ref{lm:ex_x_Ax_Bx_n0} there exists a vector $e_2\in V$ such that $B(e_2)= q A(e_2) \not=0$.

Put $e_1=A(e_2)$.
Since $\rank A=1$, it follows that $e_1$ generate the image of $A$, that is for each $y\in V$ there is $\alpha_{y}\in\mathbb{F}$ such that $A(y)=\alpha_{y} e_1$.
In particular, $e_1$ is eigen vector for $A$.
But $\spectr A=\{0\}$, whence $A(e_1)=0$.

Moreover, by (b) of Lemma~\ref{lm:collin_AB_V_W} we also have that $\rank B=1$ and from assumption (i) we get $\spectr B=\spectr A=\{0\}$.
Hence $B(e_1)=0$.
Extends the vectors $e_1,e_2$ to the basis 
$\langle e_1,e_2,f_3,\ldots,f_n\rangle$ of $V$.
Then for each $i=3,\ldots,n$ there exists $\alpha_i\in\mathbb{F}$ such that $A(f_i) = \alpha_i e_1$.
Put $e_i = f_i - \alpha_i e_2$.
Then $\langle e_1,e_2,e_3,\ldots,e_n\rangle$ is also a basis for $V$ and 
$$
A(e_i) = A(f_i) - A(\alpha_i e_2) = \alpha_i e_1 - \alpha_i e_1=0, \qquad i=3,\ldots,n.
$$
Thus in the basis $\langle e_1,\ldots,e_n\rangle$ the matrix $A$ is given by~\eqref{equ:A_zero_spectrum_B}.

Moreover, it follows from (ii) that $B(e_i)=b_i e_1$ for $i=3,\ldots,n$.
Thus in this basis $B$ is given by matrix of the form~\eqref{equ:A_zero_spectrum_B} with $q\not=0$.

Proposition~\ref{pr:collin-conjug-lin-maps} is completed.

\section{Jacobi matrices of orbit preserving diffeomorphisms} \label{sect:Jacobi_matr}
\begin{theorem}\label{th:AH_HA}
Let $\AFld$ be a vector field of $\RRR^n$ given by~\eqref{equ:VF}.
Suppose that the set $\singA$ of zeros of $\AFld$ is nowhere dense near $\orig$.
Let also $\hdif:(\RRR^n,\orig)\to(\RRR^n,\orig)$ be a germ of diffeomorphisms at $\orig$ preserving foliation by orbits of $\AFld$, i.e. there exists a neighbourhood $\Vman$ of $\orig$ such that $\hdif(\orb\cap\Vman)$ is contained in some (perhaps another) orbit $\orb'$ of $\AFld$.
Denote $$A=\nabla\AFld, \qquad  H=J(\hdif,\orig).$$

If $\rank A \geq2$ then there exist $\tau\in\RRR$ such that
$$
AH= \tau HA.
$$

Suppose $\rank A=1$, so in some local coordinates $A$ is given by one of the matrices~\eqref{equ:A_nonzero_spectrum_B} or~\eqref{equ:A_zero_spectrum_B}.
Then $$ A\cdot(HG_i) = (HG_i)\cdot A,$$
where $G_i$, $(i=1,2)$, is given by the corresponding matrices~\eqref{equ:A_nonzero_spectrum_G1} or~\eqref{equ:A_zero_spectrum_G2}.
\end{theorem}
\begin{proof}
Consider the pushforward of $\AFld$ via $\hdif$, i.e. 
\begin{equation}\label{equ:pushforward_F}
\BFld=\hdif_{*}\AFld = T\hdif \circ \AFld \circ \hdif^{-1}.
\end{equation}
Let $\BFld=(\BFld_1,\ldots,\BFld_n)$ be the coordinate functions of $\BFld$.
Then we can write 
$$
\BFld_i(x) = b_{i1}x_1 + \cdots + b_{in}x_n + o(\|x\|^2), \qquad i=1,\ldots,n,
$$
and thus $\jet(\BFld)(x)=Bx$, where 
$$B=\left(
\begin{matrix}
b_{11} & \cdots & b_{1n} \\
\cdots & \cdots & \cdots \\
b_{n1} & \cdots & b_{nn}
\end{matrix}
\right).
$$
As noted in~\eqref{equ:nabla_for_pushforward}
$$
B=H A H^{-1}.
$$

Since $\hdif$ maps every orbit of $\AFld$ into itself, it follows that $\AFld$ and $\BFld$ are collinear at each $z$ such that $\AFld(z)\not=0$.
Hence
$$
\AFld_i(z) \, \BFld_j(z) \;=\; \AFld_j(z) \, \BFld_i(z), \qquad i,j=1,\ldots,n.
$$
Therefore
\begin{multline*}
\bigl( a_{i1}x_1 + \cdots + a_{in}x_n + o(\|x\|^2) \bigr)
\bigl( b_{j1}x_1 + \cdots + b_{jn}x_n + o(\|x\|^2) \bigr)= \\ =
\bigl( a_{j1}x_1 + \cdots + a_{jn}x_n + o(\|x\|^2) \bigr)
\bigl( b_{i1}x_1 + \cdots + b_{in}x_n + o(\|x\|^2) \bigr),
\end{multline*}
whence
\begin{multline*}
\bigl( a_{i1}x_1 + \cdots + a_{in}x_n \bigr)
\bigl( b_{j1}x_1 + \cdots + b_{jn}x_n \bigr)= \\ =
\bigl( a_{j1}x_1 + \cdots + a_{jn}x_n \bigr)
\bigl( b_{i1}x_1 + \cdots + b_{in}x_n \bigr), \qquad i,j=1,\ldots,n.
\end{multline*}
This imlplies that $A(v)$ and $B(v)$ are collinear for each tangent vector $v\in T_{\orig}\RRR^n$.
Now the statement of Theorem~\ref{th:AH_HA} follows from Proposition~\ref{pr:collin-conjug-lin-maps}.
\end{proof}

\begin{corollary}\label{cor:j1_of_orb_pres_dif}
Let $\AFld$ be a \TC\ vector field, $\AFlow:\disk\times\RRR\to\disk$ be the flow of $\AFld$, $\hdif:\disk\to\disk$ be a germ of a diffeomorphism at $\orig$ preserving orbits of $\AFld$, i.e. $\hdif(\orb)=\orb$ for every orbit of $\AFld$.
Denote by $H=J(\hdif,\orig)$ the Jacobi matrix of $\hdif$ at $\orig$.
Suppsoe that $A=\nabla\AFld \not=0$.

{\rm(1)}~If $A=\left(\begin{smallmatrix}
0 & b \\-b & 0
\end{smallmatrix}\right)$ for some $b\not=0$, then $H$ coincides with one of the following matrices:
\begin{equation}\label{equ:H_A_nondeg}
\left(\begin{smallmatrix}
\cos b\omega & \sin b\omega \\ -\sin b\omega & \cos b\omega
\end{smallmatrix}\right),
\qquad
\left(\begin{smallmatrix}
\cos b\omega & \sin b\omega \\ \sin b\omega & -\cos b\omega
\end{smallmatrix}\right)
\end{equation}
for some $\omega\in\RRR$.
Hence in the first case $j^1\hdif(\orig)=j^1\AFlow_{\omega}(\orig)$, i.e. the linear parts of $\hdif$ and $\AFlow_{\omega}$ at $\orig$ coincide.

{\rm(2)}~If $A=\left(\begin{smallmatrix}
0 & b \\ 0 & 0
\end{smallmatrix}\right)$ for some $b\not=0$, then $H$ coincides with one of the following matrices:
\begin{equation}\label{equ:H_A_semideg}
\left(\begin{smallmatrix}
1 & b\omega \\ 0 & 1
\end{smallmatrix}\right),
\qquad
\left(\begin{smallmatrix}
-1 & b\omega \\ 0 & -1
\end{smallmatrix}\right),
\qquad
\left(\begin{smallmatrix}
-1 & b\omega \\ 0 & 1
\end{smallmatrix}\right),
\qquad
\left(\begin{smallmatrix}
1 & b\omega \\ 0 & -1
\end{smallmatrix}\right)
\end{equation}
for a certain unique $\omega\in\RRR$.
Again in the first case $j^1\hdif(\orig)=j^1\AFlow_{\omega}(\orig)$.
\end{corollary}
\begin{proof}
(1) Suppose 
$A=\left(\begin{smallmatrix}
0 & b \\-b & 0
\end{smallmatrix}\right)$ 
for some $b\not=0$.
Then the Jacobi matrix of $\AFlow_t$ at $\orig$ is given by
$$
J(\AFlow_t,\orig) =
\left(\begin{smallmatrix}
\cos bt  & \sin bt \\ -\sin bt & \cos bt 
\end{smallmatrix}\right).
$$

Let $H = J(\hdif,\orig) = 
\left(\begin{smallmatrix}
\alpha & \beta \\ \gamma & \delta
\end{smallmatrix}\right).$
By Theorem~\ref{th:AH_HA} $A=\tau HAH^{-1}$ for some $\tau\not=0$.
Since $A$ is non-degenerate, we obtain 
$\det(A) = \det (\tau HAH^{-1}) = \tau^2 \det(A)$, whence 
$\tau=\pm1$.

A direct calculation shows that $\gamma=\tau\beta$ and $\delta=-\tau\alpha$, so
$$
H=\left(\begin{smallmatrix}
\alpha & \beta \\ -\tau\beta & -\tau\alpha
\end{smallmatrix}\right) = 
(\alpha^2+\beta^2)
\left(\begin{smallmatrix}
\cos b\omega & \sin b\omega \\ -\tau \sin b\omega & -\tau\cos b\omega
\end{smallmatrix}\right)
$$
for some $\omega\in\RRR$.
We claim that $\alpha^2+\beta^2=1$.
This will imply that $H$ coincides with one of the matrices~\eqref{equ:H_A_nondeg}.

Indeed, it is easy to verify that the eigen values of $H$ are given by 
$$
\mu_{1,2} = 
\begin{cases}
 (\alpha^2+\beta^2)e^{\pm iq}, & \tau=1 \\
 \pm\sqrt{\alpha^2+\beta^2}, & \tau=-1.
\end{cases}
$$
By Lemma~\ref{lm:eigen_val_H} $|\mu_1|=|\mu_2|=1$, whence $\alpha^2+\beta^2=1$.

(2) Suppose 
$A=\left(\begin{smallmatrix}
0 & b \\ 0 & 0 
\end{smallmatrix}\right)$,
for some $b\not=0$.
Then the Jacobi matrix of $\AFlow_t$ at $\orig$ is given by
$$
J(\AFlow_t,\orig) =
\left(\begin{smallmatrix}
1 &  bt \\ 0 & 1 
\end{smallmatrix}\right).
$$
On the other hand, by (A3) of Proposition~\ref{pr:collin-conjug-lin-maps} $A$ commutes with matrix $K=HG$, where 
$$
G = \left(\begin{smallmatrix}
1 &  0 \\ 0 & 1/q 
\end{smallmatrix}\right)
$$
is given by~\eqref{equ:A_zero_spectrum_G2} for some $q\not=0$.

Write 
$K=\left(\begin{smallmatrix}
\alpha & \beta \\ \gamma & \delta
\end{smallmatrix}\right)$.
Then the identity $AK=KA$ easily implies that $\delta=\alpha$ and $\gamma=0$, so
$K=\left(\begin{smallmatrix}
\alpha & \beta \\ 0 & \alpha
\end{smallmatrix}\right)$.
Hence 
$$
H = K G^{-1} = 
\left(\begin{smallmatrix}
\alpha & \beta \\ 0 & \alpha
\end{smallmatrix}\right)
\left(\begin{smallmatrix}
1 & 0 \\ 0 & q
\end{smallmatrix}\right) 
\ = \
\left(\begin{smallmatrix}
\alpha & \beta q \\ 0 & \alpha q
\end{smallmatrix}\right)
$$
Evidently, $\alpha$ and $\alpha q$ are eigen values of $H$.
Then by Lemma~\ref{lm:eigen_val_H} $|\alpha|=|q|=1$, whence $H$ coincides with one of the matrices~\eqref{equ:H_A_semideg}, where $\omega=\beta q/b$.
\end{proof}

Let $\AFld$ be a \TC\ vector field on $\disk$.
Denote by $\EAp{1}$ the subset of $\EApl$ consisting of maps $\hdif$ such that $\jo{\hdif}=\id$.

Suppose that $A=\nabla\AFld\not=0$.
Then by Corollary~\ref{cor:j1_of_orb_pres_dif} for each $\hdif\in\EApl$, there exist $\omega\in\RRR$ such that $J(\hdif,\orig)=J(\AFlow_{\omega},\orig)$.
Hence $$\AFlow_{-\omega}\circ\hdif \ \in \ \EAp{1}.$$
\begin{proposition}\label{pr:sect_ShA_j1}
Suppose that $A=\nabla\AFld\not=0$.
Then there exists a neighbourhood $\Nbh$ of the identity map $\id_{\disk}$ in $\EApl$ and a $\Wr{1}$-function $\omega:\Nbh\to\RRR$ such that $\jo{\hdif}=\jo{\AFlow_{\omega(\hdif)}}$ for each $\hdif\in\EApl$ and $\omega(\id_{\disk})=0$.
Hence we have a well-defined map
$$
H:\EApl \supset \Nbh \to \EAp{1},
\qquad
H(\hdif) = \AFlow_{-\omega(\hdif)} \circ \hdif.
$$
This function preserves smoothness.
\end{proposition}
\begin{proof}
(1) Assume that $A=\left(\begin{smallmatrix}
0 & b \\ -b & 0 
\end{smallmatrix}\right)$.
Let $p:\RRR\to SO(2)$ be the covering map given by 
$p(t)=\left(\begin{smallmatrix}
\cos 2\pi t & \sin 2\pi t \\
-\sin 2\pi t & \cos 2\pi t
\end{smallmatrix}\right)$ for $t\in\RRR$.
Then by Corollary~\ref{cor:j1_of_orb_pres_dif} for each $\hdif\in\EApl$ we have two maps
$$
\begin{CD}
\EApl @>{j^1:\hdif\,\mapsto\, J(\hdif,\orig)}>> SO(2) @<{p}<< \RRR
\end{CD}
$$
Since $p$ is a local diffeomorphism, there exists a $\Wr{1}$-neighbourhood $\Nbh$ of the identity map $\Nbh$ in $\EApl$ on which the composition $$\omega=p^{-1}\circ j^1:\Nbh\to\RRR$$ is well-defined and satisfies $\omega(\id_{\disk})=0$.

If $\hdif_{\tau}$ is a family of maps in $\EApl$ smoothly depending on $\tau\in\RRR^n$, then evidently so does $\omega(\hdif_{\tau})$.
This implies that $\omega$ preserves smoothness.

(2) Suppose $|A|=0$ but $A\not=0$.
Then by Lemma~\ref{lm:1jet_TCVF} we can assume that 
$A=\left(\begin{smallmatrix}
0 & b \\ 0 & 0 
\end{smallmatrix}\right)$,
whence by Corollary~\ref{cor:j1_of_orb_pres_dif} for each $\hdif\in\EApl$ its Jacobi matrix at $\orig$ has one of the following forms
\begin{equation}\label{equ:alternative_for_H}
\text{either} \quad  
\left(\begin{smallmatrix}
1 & b\omega \\ 0 & 1 
\end{smallmatrix}\right),
\quad \text{or} \quad
\left(\begin{smallmatrix}
-1 & b\omega \\ 0 & -1 
\end{smallmatrix}\right).
\end{equation}
Let $\Nbh$ be a subset of $\EApl$ consisting of $\hdif$ for which 
$J(\hdif,\orig)=\left(\begin{smallmatrix}
1 & b\omega \\ 0 & 1 
\end{smallmatrix}\right)$.
It follows from~\eqref{equ:alternative_for_H} that $\Nbh$ is $\Wr{1}$ open in $\EApl$.

Define $p:\RRR\to GL(2,\RRR)$ by $p(t) = \left(\begin{smallmatrix}
1 & b t \\ 0 & 1 
\end{smallmatrix}\right)$ and let $G=p(\RRR)$ be the image of $p$.
Then $p:\RRR\to G$ is a diffeomorphism.
Now similarly to the case (1) for each $\hdif\in\Nbh$ we have the following two maps
$$
\begin{CD}
\Nbh @>{j^1:\hdif\,\mapsto\, J(\hdif,\orig)}>> G @<{p}<< \RRR.
\end{CD}
$$
Then the inverse map 
$\omega=p^{-1}\circ j^1:\Nbh\to\RRR$ is well-defined.
Evidently, it preserves smoothness and satisfies $\omega(\id_{\disk})=0$.
\end{proof}

\section{Polar coordinates}\label{sect:polar_coordinates}
Consider the plane $\RRR^2$ with coordinates $(\phi,\rrho)$.
Let $\IHman=\{\rrho>0\}$ be the open upper half-plane, $\Hman=\{\rrho\geq0\}$ be its closure, and $P:\Hman\to\RRR^2$ be the map defined by
$$
P(\phi,\rrho)=(\rrho\cos\phi,\rrho\sin\phi).
$$
Thus $(\phi,\rrho)$ are the polar coordinates in $\RRR^2$.

Identifying $\RRR^2$ with $\CCC$ we can also define $P:\Hman\to\CCC$ by $$P(\phi,\rrho)=\rrho e^{i\phi}.$$

Let $\eta:\Hman\to\Hman$ be given by $\eta(\phi,\rrho)=(\phi+2\pi,\rrho)$.
Then $P=P\circ\eta$ and thus $\eta$ induces a $P$-equivariant $\ZZZ$-action $\Hman$ such that $P$ is a factor map $\Hman\to\Hman/\ZZZ \equiv \RRR^2$.

In this section we recall how to lift certain objects like functions, self-maps, and vector fields defined on $\RRR^2$ to the corresponding objects on $\Hman$ via $P$.

\subsection{Correspondence between flat functions}
Let $\CiZHR$ be the subset of $\Ci{\Hman}{\RRR}$ consisting of all $\ZZZ$-invariant functions $\tafunc:\Hman\to\RRR$, i.e. $$\tafunc(\phi+2\pi,\rrho)=\tafunc(\phi,\rrho).$$

Let also $\FlatZHR \subset \CiZHR$ be the subset consisting of all function flat on $\dHman$ and 
$\FlatOR\subset\Ci{\RRR^2}{\RRR}$ be the subset consisting of all function flat at the origin $\orig$.

Notice that for every $\tafunc\in\CiZHR$ there exists a unique $\Cinf$ function $\afunc:\RRR^2\setminus\{0\}\to\RRR$ such that $\tafunc=\afunc\circ P$.
But in general, $\afunc$ can not be even continuously extended to all of $\RRR^2$.
For instance, $\afunc$ is continuous if and only if $\tafunc$ is constant on the $\phi$-axis $\dHman=\{\rrho=0\}$.

The following statement shows that $\afunc$ is actually $\Cinf$ and flat at $\orig$ whenever $\tafunc$ is flat on $\dHman$.
Smoothness and flatness of $\afunc$ was shown in~\cite[Lm.\;5.3]{Takens:AIF:1973}, and in~\cite[Th.\;5.1]{Maks:Hamvf:2006} continuity of the correspondence $\tafunc\mapsto\afunc$ was treated.
\begin{lemma}\label{lm:flat_Zinv_func}{\rm\cite[Lm.\;5.3]{Takens:AIF:1973}, \cite[Th.\;5.1]{Maks:Hamvf:2006}.}
There exists a $\contW{\infty}{\infty}$-continuous and preserving smoothness map 
$$
\ff:\FlatZHR\to\FlatOR
$$
such that $\tafunc = \ff(\tafunc)\circ P$ for all $\tafunc\in\FlatZHR$.
Thus every smooth, $\ZZZ$-invariant, flat on $\dHman$ function on $\Hman$ induces a $\Cinf$ function on $\RRR^2$ flat at $\orig$.
\end{lemma}

\subsection{Correspondence between smooth maps}
Let $\hdif:\RRR^2\to\RRR^2$ be a continuous map such that $\hdif^{-1}(0)=0$.
Then $\hdif$ lifts to a certain map $\thdif:\IHman\to\IHman$ which commutes with $\eta$ and satisfies 
$\hdif\circ P = P \circ\thdif$.
Hence for each $n\in\ZZZ$ the map $\thdif\circ\eta$ is also a lifting of $\hdif$.

It is well-known that if $\hdif$ is at least of class $C^{1}$, then $\thdif$ extends to a continuous map $\thdif:\Hman\to\Hman$.
Moreover, if $\hdif$ is $\Cinf$, then so is $\thdif$.

Our aim is to estimate continuity of the correspondence $\hdif\mapsto\thdif$.

Let $\MapR$ be the subset of $\Ci{\RRR^2}{\RRR^2}$ consisting of maps $\hdif:\RRR^2\to\RRR^2$ such that 
\begin{itemize}
 \item $\hdif(\orig)=\orig$;
 \item $\hdif(\RRR^2\setminus\orig)\subset \RRR^2\setminus\orig$;
 \item $\jo{\hdif}=\tau \cdot \id$ \ for some $\tau >0$.
\end{itemize}

Let also $\MapH$ be the subset of $\Ci{\Hman}{\Hman}$ consisting of maps $\thdif:\Hman\to\Hman$ such that 
\begin{itemize}
 \item $\thdif$ is fixed on $\dHman$;
 \item $\thdif(\IHman)\subset\IHman$.
\end{itemize}

Finally let $\MapZH$ be the subset of $\MapH$ consisting of $\ZZZ$-equivariant maps, i.e. maps commuting with $\ZZZ$-action on $\Hman$. 
In other words, $\thdif=(\thp,\thr)\in\MapH$ belongs to $\MapZH$ if and only if
$$
\thp(\phi+2\pi,\rrho)=\thp(\phi,\rrho)+2\pi,
\qquad
\thr(\phi+2\pi,\rrho)=\thr(\phi,\rrho).
$$

The following statement claims that every $\hdif\in\MapR$ lifts to a unique $\thdif\in\MapZH$ such that $P\circ\thdif=\hdif\circ P$ and the correspondence $\hdif\mapsto\thdif$ is $\contS{\infty}{\infty}$-continuous.
\begin{proposition}\label{pr:lift_Map_id}
There exists a unique mapping $$\ml:\MapR\to\MapZH$$ such that 
$P\circ\ml(\hdif)=\hdif\circ P$ for $\hdif\in\MapR$.
This map preserves smoothness and is $\contS{k+1}{k}$-continuous for each $k\geq0$.
\end{proposition}
\begin{proof}
Let $\hdif=(\hx,\hy) \in\MapR$.
We have to construct a $\Cinf$ map $\thdif:\Hman\to\Hman$ fixed on $\dHman$ and such that $P\circ\thdif=\hdif\circ P$.

By assumption $\jo{\hdif}=\tau\,\id$ for some $\tau>0$, i.e.

\begin{equation}\label{equ:2-jet_of_hdif}
\begin{array}{rcl}
\hx(x,y) &=& \tau x + x\alpha_{1} + y\alpha_{2}, \\ [2mm]
\hy(x,y) &=& \tau y + x\beta_{1} + y\beta_{2}
\end{array}
\end{equation}
for some $\Cinf$ functions $\alpha_{i},\beta_{i}:\RRR^2\to\RRR$ such that $\alpha_{i}(\orig)=\beta_{i}(\orig)=0$.
These functions are not unique and, for instance, they can be given by the following formulas, e.g.~\cite{Hirsch:DiffTop}:
\begin{equation}\label{equ:afunc_i__bfunc_i}
\begin{array}{lcl}
\afunc_1(x,y) = \int\limits_{0}^{1}\ddd{\hx}{x}(tx,y)dt-\tau, 
&\ & 
\afunc_2(x,y) = \int\limits_{0}^{1}\ddd{\hx}{x}(0,ty)dt.\\ [5mm]
\bfunc_1(x,y) = \int\limits_{0}^{1}\ddd{\hy}{x}(tx,y)dt,
&\ &
\bfunc_2(x,y) = \int\limits_{0}^{1}\ddd{\hy}{x}(0,ty)dt - \tau.
\end{array}
\end{equation}

Hence 
$$ 
\begin{array}{rcl}
\hx\circ P(\phi,\rrho) &=& \rrho\cos\phi + \rrho A(\phi,\rrho),\\ [2mm]
\hy\circ P(\phi,\rrho) &=& \rrho\sin\phi + \rrho B(\phi,\rrho),
\end{array}
$$ 
where 
\begin{equation}\label{equ:hdif_P}
\begin{array}{rcl}
A(\phi,\rrho)&=&\cos\phi \cdot \alpha_{1}\circ P(\phi,\rrho)+\sin\phi \cdot \alpha_{2}\circ P(\phi,\rrho),\\[2mm]
B(\phi,\rrho)&=&\cos\phi \cdot \beta_{1}\circ P(\phi,\rrho)+\sin\phi \cdot \beta_{2}\circ P(\phi,\rrho).
\end{array}
\end{equation}
It follows that $A$ and $B$ are $\Cinf$, $\ZZZ$-invariant, and vanish on $\dHman$, i.e. $A(\phi,0)=B(\phi,0)=0$ for all $\phi\in\RRR$.

Since $P:\IHman\to\RRR^2\setminus\{\orig\}$ is a covering map, there exists a $\Cinf$ map
$$
\tgdif=(\thp,\thr):\IHman\to\IHman
$$
such that $P\circ \tgdif = \hdif\circ P$.
In other words,
$$
\begin{array}{rcl}
\thr(\phi,\rrho)\,\cos\thp(\phi,\rrho) &=& \hx\circ P(\phi,\rrho), \\ [2mm]
\thr(\phi,\rrho)\,\sin\thp(\phi,\rrho) &=& \hy\circ P(\phi,\rrho)
\end{array}
$$
for $\rrho>0$.
This map is not unique, commutes with $\eta$, and can be replaced by $\tgdif\circ\eta^n$ for any $n\in\ZZZ$.

Then for each $\phi_0\in\RRR$ we get that 
$$ 
\begin{array}{rcl}
\thr\,\cos(\thp-\phi_0) &=& \hx\circ P \cdot \cos\phi_0 + \hy\circ P \cdot \sin\phi_0 = \\ [1mm]
&=& \rrho\cos(\phi-\phi_0)+ \rrho A_1(\phi,\rrho), \\ [3mm]
\thr\,\sin(\thp-\phi_0) &=& \hy\circ P \cdot \cos\phi_0 - \hx\circ P \cdot \sin\phi_0 = \\ [1mm]
&=& \rrho\sin(\phi-\phi_0)+ \rrho B_1(\phi,\rrho),
\end{array}
$$ 
where 
\begin{equation}\label{equ:RPhi_phi0}
A_1 = A \cos\phi_0 + B\sin\phi_0,
\qquad
B_1 = B \cos\phi_0 - A\sin\phi_0.
\end{equation}
Again $A_1,B_1:\Hman\to\RRR$ are $\Cinf$, $\ZZZ$-invariant, and vanish for $\rrho=0$.

Consider the following two functions:
\begin{equation}\label{equ:thdif_phi0}
\begin{array}{rcl}
\thp_{\phi_0}(\phi,\rrho) & = & 
\phi_0 + \arctg\, \dfrac{\sin(\phi-\phi_0)+ \rrho B_1(\phi,\rrho)}{\cos(\phi-\phi_0)+ \rrho A_1(\phi,\rrho)}, \\ [6mm]
\thr_{\phi_0}(\phi,\rrho) & = &
\dfrac{\rrho\cos(\phi-\phi_0) + \rrho A_1(\phi,\rrho)}{\cos(\thp_{\phi_0}(\phi,\rrho)-\phi_0)}.
\end{array}
\end{equation}
Evidently, there exist $a,b>0$ such that $\thp_{\phi_0}$ and $\thr_{\phi_0}$ are well-defined and $\Cinf$ on the following neighbourhood
$$\Vman_{\phi_0}=[\phi_0-a,\phi_0+a]\times[0,b]$$
of $(\phi_0,0)$ in $\Hman$.

Moreover, the map $\thdif_{\phi_0}=(\thp_{\phi_0},\thr_{\phi_0}):\Vman_{\phi_0}\to\Hman$ is a lifting of $\hdif$ in the sense that $P\circ\thdif_{\phi_0}=\hdif\circ P: P^{-1}(\Vman_{\phi_0}) \to \RRR^2$.

Also notice that $\thdif_{\phi_0}(\phi,0)=(\phi,0)$, i.e. $\thdif_{\phi_0}$ is fixed on $\dHman\cap\Vman_{\phi_0}$.
Therefore for distinct $\phi_0,\phi_1\in\RRR$ the lifting $\thdif_{\phi_0}$ and $\thdif_{\phi_1}$ coincide on $\Vman_{\phi_0}\cap\Vman_{\phi_1}\cap\dHman$ whenever this intersection is non-empty.
Hence $\thdif_{\phi_0}=\thdif_{\phi_1}$ on all of $\Vman_{\phi_0}\cap\Vman_{\phi_1}$.

This implies that the partial maps $\thdif_{\phi_0}$, $(\phi_0\in\RRR)$, define a unique $\Cinf$ lifting $\thdif$ of $\hdif$ fixed on $\dHman$.

Finally, it follows from~\eqref{equ:thdif_phi0} and $\ZZZ$-invariantness of $A_1$ and $B_1$ that $\thdif$ is $\ZZZ$-equivariant.
Thus $\thdif\in\MapZH$ and we define 
$$
\ml(\hdif) = \thdif.
$$

It remains to verify $\contS{k+1}{k}$-continuity of $\ml$.

Consider the following sequence of correspondences:
$$
\hdif 
\;\stackrel{\eqref{equ:afunc_i__bfunc_i}}{\longmapsto}\;
(\alpha_{ij},\beta_{ij})_{i,j=1,2}
\;\stackrel{\eqref{equ:hdif_P}}{\longmapsto}\;
(A,B)
\;\stackrel{\eqref{equ:RPhi_phi0}}{\longmapsto}\;
(A_1,B_1)
\;\stackrel{\eqref{equ:thdif_phi0}}{\longmapsto}\;
\thdif_{\phi_0}.
$$
The first one expresses $\afunc_{i}$ and $\bfunc_{i}$ via the first partial derivatives of the coordinate functions of $\hdif$.
All others are just compositions of smooth maps.
It follows that for each integer $k\geq0$ there exists a constant $C=C(\phi_0,k)>0$ depending only on $\phi_0$ and $k$ such that 
\begin{equation}\label{equ:lift_cont_estim_phi_0}
\| \thdif'_{\phi_0} - \thdif_{\phi_0} \|^{k}_{\Vman_{\phi_0}} \;\leq\; C \;
\| \hdif' - \hdif \|^{k+1}_{P(\Vman_{\phi_0})},
\end{equation}
for every $\hdif'\in\MapR$ and its lifting $\hdif'_{\phi_0}$.

\begin{claim}\label{clm:estimate_KL}
Let $K\subset\Hman$ be a compact subset and $k\geq0$ and $L=P(K)$.
Then there exists $C>0$ all depending on $K$ and $k$ such that 
$$
\| \thdif' - \thdif \|^{k}_{K} < C \| \hdif' - \hdif \|^{k+1}_{L}
$$
for each $\hdif'\in\EBZpl$.
Hence $\ml$ is $\contW{k+1}{k}$-continuous.
\end{claim}
\begin{proof}
Put \ $\tilde K = [0,2\pi]\cap [0,+\infty)\,\, \bigcap\,\, \mathop\cup\limits_{n=-\infty}^{+\infty} \eta^n(K)$, \  see Figure~\ref{fig:cont_lift}.
Since $\thdif$ is $\ZZZ$-invariant, we have that 
$$\| \thdif' - \thdif \|^{k}_{K} =\| \thdif' - \thdif \|^{k}_{\tilde K}$$
for any $\thdif'\in\MapZH$, where $\|\cdot\|^{k}_{K}$ is the usual $C^{k}$-norm of a map on a compact set $K$, see~\cite{Hirsch:DiffTop}.
\begin{figure}[ht]
\centerline{\includegraphics[height=2cm]{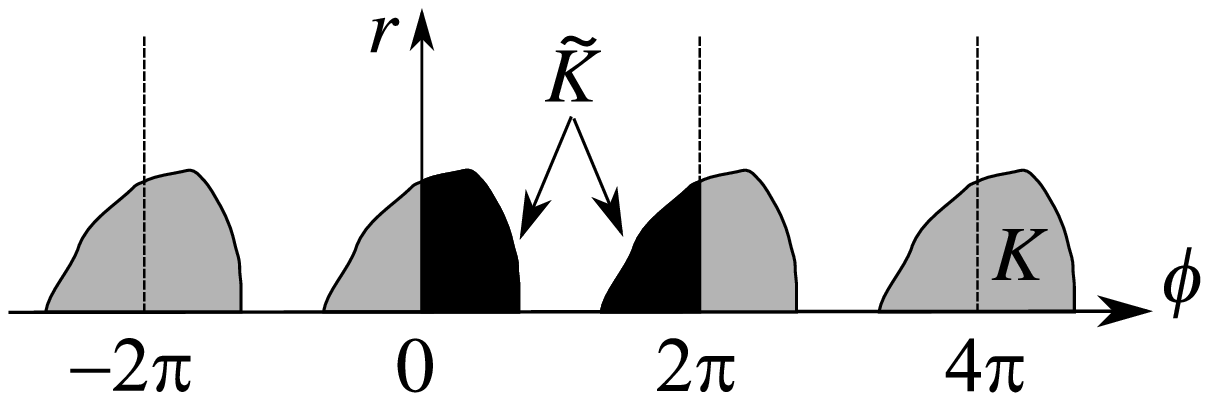}}
\caption{}\label{fig:cont_lift}
\end{figure}

Therefore it suffices to consider the case when $K\subset [0,2\pi]\cap [0,+\infty)$.

Then there exist finitely many values $\phi_{1},\ldots,\phi_{m}\in[0,2\pi]$ such that 
$$
K \cap \dHman \; \subset \; \mathop\cup\limits_{i=1}^{m} \Vman_{\phi_{i}}.
$$
Denote 
$K_0= K \cap \mathop\cup\limits_{i=1}^{m} \Vman_{\phi_{i}}$, and $K_1 = \overline{K \setminus \Vman}$.

Put $L=P(K)$, $L_0=P(K_0)$, $L_1=P(K_1)$.
Then by~\eqref{equ:lift_cont_estim_phi_0} there exists $C_0>0$ such that 
\begin{equation}\label{equ:lift_cont_estim_phi_V}
\| \thdif' - \thdif \|^{k}_{K_0} \;\leq\; C_0 \;
\| \hdif' - \hdif \|^{k+1}_{L_0},
\end{equation}
Moreover, since $P$ homeomorphically maps $K_1$ onto $L_1$, there exists $C_1>0$ such that 
\begin{equation}\label{equ:lift_cont_estim_phi_K_1}
\| \thdif' - \thdif \|^{k}_{K_1} \;\leq\; C_1 \;
\| \hdif' - \hdif \|^{k}_{L_1},
\end{equation}
Put $C=\max\{C_0,C_1\}$. 
Then 
$$
\| \thdif' - \thdif \|^{k}_{K} \;\leq\; C \;\| \hdif' - \hdif \|^{k+1}_{L}.
$$
Claim~\ref{clm:estimate_KL} is proved.
\end{proof}

To prove $\contS{k+1}{k}$-continuity of $\ml$ take any countable locally finite cover $\mathcal{K}=\{K_i\}_{i\in\NNN}$ of the strip $S=[0,2\pi]\times[0,+\infty)$ by compact subsets.
Denote $K^{j}_{i}= \eta^{j}(K_i)$ and $L_i=P(K_i)$.
Then 
$$
\widetilde{\mathcal{K}} = \{K^{j}_{i}\}_{i\in\NNN, j\in\ZZZ}
$$
is a locally finite cover of $\Hman$ and 
$\mathcal{L}=\{L_i\}_{i\in\NNN}$ is a locally finite cover of $\RRR^2$.

For each $i\in\NNN$ take any $\eps_i>0$ and let 
Denote $$\widetilde{\Nbh} = \{\thdif'\in\EBZpl \ : \  \| \thdif' - \thdif \|^{k}_{K^j_i}< \eps_i \}.$$
Then $\widetilde{\Nbh}$ is a base $\Sr{k}$-neighbourhood of $\thdif$ in $\EBZpl$.

It now follows from Claim~\ref{clm:estimate_KL} that for each $i\in\NNN$ there exists $\delta_i>0$ such that 
if $\| \hdif' - \hdif \|^{k+1}_{L_i}<\delta_i$ for $\hdif'\in\EApl$, then $\| \thdif' - \thdif \|^{k}_{K^j_i} <\eps_i$ for any $j\in\ZZZ$.

Notice that $$\Nbh = \{\hdif'\in\EApl \ : \  \| \hdif' - \hdif \|^{k+1}_{L_i}< \delta_i \}$$
is a $\Sr{k+1}$-neighbourhood of $\hdif$ in $\EApl$, and $\ml(\Nbh) \subset\widetilde{\Nbh}$.
This implies  $\contS{k+1}{k}$-continuity of $\ml$.
\end{proof}

\subsection{Lifting of vector fields.}
Let $\AFld=\AFld_x\dd{x}+\AFld_y\dd{y}$ be a vector field on $\RRR^2$.
Since $P$ is a local diffeomorphism of $\Hman$ onto $\RRR^2\setminus\{\orig\}$, $\AFld$ induces a certain vector field 
$\BFld=\BFld_{\phi}\dd{\phi} + \BFld_{\rrho}\dd{\rrho}$ on $\IHman$ such that $TP\circ\BFld=\AFld\circ P$, i.e. the following diagram is commutative:
$$
\begin{CD}
T\IHman @>{TP}>> T(\RRR^2\setminus\orig) \\
@A{\BFld}AA @AA{\AFld}A  \\
\IHman @>{P}>> \RRR^2\setminus\orig
\end{CD}
$$

It easily follows, see e.g.~\cite{Maks:Hamvf:2006}, that coordinate functions on $\AFld$ and $\BFld$ are related by the following identity:
\begin{equation}\label{equ:polar_expression}
\left(\begin{matrix} \BFld_{\phi}  \\ \BFld_{\rrho} \end{matrix} \right) =
\left(\begin{matrix} -\frac{1}{\rrho}\sin\phi & \frac{1}{\rrho}\cos\phi \\ \cos\phi & \sin\phi \end{matrix} \right)
\left(\begin{matrix} \AFld_{x}\circ P  \\ \AFld_{y}\circ P \end{matrix} \right).
\end{equation}
We will say that $\BFld$ is the \myemph{expression of $\AFld$ in polar coordinates} or \myemph{the lifting of $\AFld$}.

\section{Vector fields on $\RRR^2$ with a ``rotation as $1$-jet''}\label{sect:Takens_nf}

\subsection{Normal forms}
We recall here the results of F.\;Takens~\cite{Takens:AIF:1973} about normal forms of singularities of vector fields on $\RRR^2$ with a ``rotation as $1$-jet''.

Let $\AFld=\AFld_1\dd{x}+\AFld_2\dd{y}$ be a vector field on $\disk$ and
\begin{equation}\label{equ:1jet-AFld}
\nabla\AFld = \left(\begin{array}{cc}
\ddd{\AFld_1}{x} &  \ddd{\AFld_1}{y} \\ [2mm]
\ddd{\AFld_2}{x} &  \ddd{\AFld_2}{y}
\end{array} \right).
\end{equation}
Denote by $\lambda_1$ and $\lambda_2$ the eigen values of $\nabla\AFld$ at $\orig$.
\begin{theorem}{\rm\cite{Takens:AIF:1973}.}
Suppose that $\lambda_1,\lambda_2$ are non-zero purely imaginary, i.e. $\lambda_{1,2}=\pm i\nu$ for some $\nu>0$.
Then there exists a diffeomorphism $\gdif:(\disk,\orig)\to(\disk,\orig)$ such that either
\begin{equation}\label{equ:Takens_flat}
\gdif_{*}\BFld(x,y) = \bfunc(x^2+y^2) \left(-(y+\dAy) \dd{x} + (x+\dAx) \dd{y} \right),
\end{equation}
where $\bfunc$ is a $\Cinf$ function, $\bfunc(0,0)\not=0$, $\dAx,\dAy\in\FlatOR$,
or 
\begin{multline}\label{equ:Takens_nonflat}
\gdif_{*}\BFld(x,y) = \bfunc(x,y)\left[ 
-2\pi y \frac{\partial}{\partial x} + 2\pi x \frac{\partial}{\partial y} + \right. \\ \left. +
\left(\delta(x^2+y^2)^k + \alpha(x^2+y^2)^{2k}\right)
\left( x\frac{\partial}{\partial x}+y\frac{\partial}{\partial y} \right)
\right],
\end{multline}
where $\func$ is a $\Cinf$ function, $\bfunc(0,0)=1$, $\delta=\pm1$, $k\in\NNN$, and $\alpha\in\RRR$.

Vector fields and~\eqref{equ:Takens_flat} and~\eqref{equ:Takens_nonflat} are not $\Cinf$ equivalent.
\end{theorem}

For instance the following vector field 
$$
\AFld(x,y) = -\nu y \dd{x} + \nu x\dd{y} +  \ \cdots \ \text{terms of high order} \ \cdots.
$$
satisfies assumptions of this theorem.

$\Cinf$-equivalence classes of such vector fields were studied in~\cite{Belitsky:FA:85}.

\subsection{Expressions in polar coordinates}
Our aim is to show that a vector field of type~\eqref{equ:Takens_nonflat} does not admit a strong first integral, see Corollary~\ref{cor:Takens_nf_polar}.
We need for this the expressions of~\eqref{equ:Takens_flat} and~\eqref{equ:Takens_nonflat} in polar coordinates, see also~\cite{ArnoldYllyashenko}.
\begin{lemma}\label{lm:Takens_nf1_polar}
Let 
\begin{equation}\label{equ:Takens_nf1_xy}
\AFld_1(x,y) = -(y+\dAy)\dd{x} + (x+\dAx) \dd{y},
\end{equation}
where $\dAx,\dAy\in\FlatOR$.
Then its expression in polar coordinates is the following vector field:
\begin{equation}\label{equ:Takens_nf1_polar}
\BFld_1(\phi,\rrho) = (1+\dAp)\dd{\phi} + \dAr\dd{\rrho},
\end{equation}
where $\dAp,\dAr\in\FlatZHR$.
\end{lemma}
\begin{lemma}\label{lm:Takens_nf2_polar}
Let 
\begin{multline*}
\AFld_2(x,y) = 
-2\pi y \frac{\partial}{\partial x} + 2\pi x \frac{\partial}{\partial y} +  \\  +
\left(\delta(x^2+y^2)^k + \alpha(x^2+y^2)^{2k}\right)
\left( x\frac{\partial}{\partial x}+y\frac{\partial}{\partial y} \right),
\end{multline*}
where $\delta=\pm1$, $k\in\NNN$, and $\alpha\in\RRR$.
Then its expression in polar coordinates is the following vector field:
\begin{equation}\label{equ:Takens_nf2_polar}
\BFld_2(\phi,\rrho) = 2\pi \dd{\phi} + \rrho^{2k+1}(\delta+\alpha\rrho^{2k})\dd{\rrho}.
\end{equation}
\end{lemma}
These lemmas are direct consequences of~\eqref{equ:polar_expression} and we leave them for the reader.

It follows that $\BFld_1$ and $\BFld_2$ are $\Cinf$ on all of $\Hman$, tangent to $\dHman$, and have no singular points.

\begin{corollary}\label{cor:Takens_nf_polar}
The origin $\orig$ is the limit point for each orbit of $\AFld_2$ passing sufficiently close to $\orig$.
Hence every function $\func:\RRR^2\to\RRR$ which is constant along orbits of $\AFld$ must be constant near $\orig$.
\end{corollary}
\begin{proof}
Denote $\afunc(\phi,\rrho)=\delta+\alpha\rrho^{2k}$.
Then $\afunc$ is non-zero on some neighbourhood of $\dHman$, whence the following vector field
$$
\BFld=\frac{1}{\afunc}\BFld_2 = \frac{2\pi}{\afunc} \dd{\phi} + \rrho^{2k+1}\dd{\rrho}.
$$ 
is $\Cinf$ and has the same orbit structure as $\BFld_2$.
\begin{figure}[ht]
\centerline{\includegraphics[height=2cm]{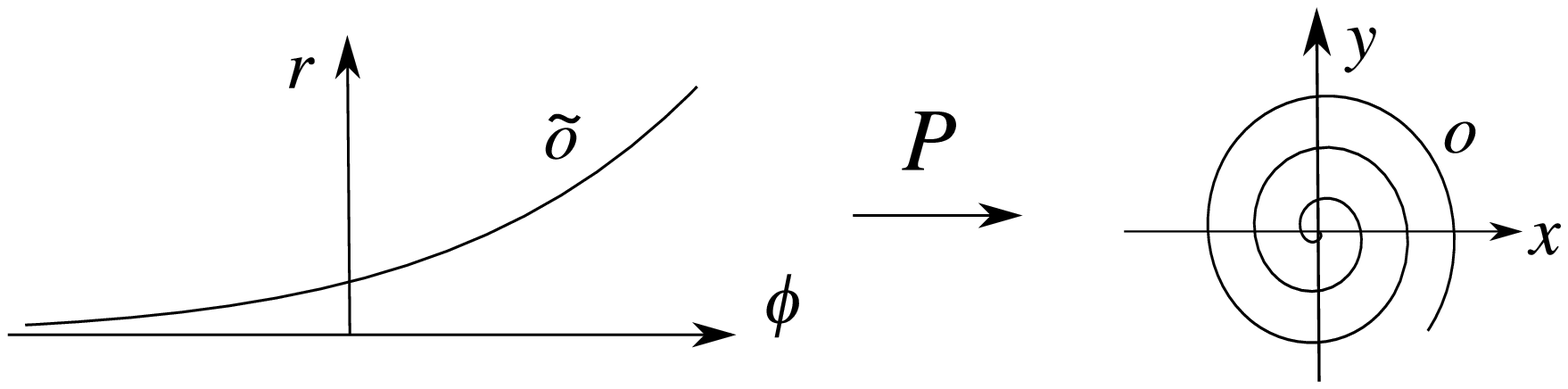}}
\caption{}\label{fig:orb_case2}
\end{figure}
It is now evident, that if $\rrho>0$ is sufficiently small, then for each $\phi\in\RRR$, the orbit $\torb$ of the point $\tz=(\phi,\rrho)$ tends to $\dHman$ when $t\to-\infty$.
Therefore the orbit $\orb=P(\torb)$ of the point $P(\tz)=\rrho e^{i\phi}$ tends to $\orig$, see Figure~\ref{fig:orb_case2}.
\end{proof}

\subsection{Flat perturbations}
Let $\BFld$ be a vector field on $\Hman$ given by~\eqref{equ:Takens_nf1_polar}:
$$
\BFld(\phi,\rrho) = (1+\dAp)\,\dd{\phi} + \dAr\,\dd{\rrho},
$$
where $\dAp,\dAr\in\FlatHR$.
Then it is easy to see that $\BFld$ is tangent to $\dHman$ and has no singular points on some neighbourhood of $\dHman$.

For simplicity assume that $\BFld$ generates a global flow $$\BFlow=(\BP,\BR):\Hman\times\RRR\to\Hman.$$

\begin{lemma}\label{lm:sect_ShB}
Suppose that $\thdif=(\thp,\thr)\in\EBpl$ has a shift function $\tsigma$ with respect to $\BFld$.
Then 
\begin{gather}
\label{equ:coord_func_shift_P}
\thp(\phi,\rrho) = \phi + \tsigma(\phi,\rrho) + \xi_{\phi}(\phi,\rrho),\\
\label{equ:coord_func_shift_R}
\thr(\phi,\rrho) = \rrho + \xi_{\rrho}(\phi,\rrho),
\end{gather}
where $\xi_{\phi},\xi_{\rrho}\in\FlatHR$.

The coordinate functions of $\BFlow$ have the following form:
\begin{gather}
\label{equ:BFlow_t_P}
\BP(\phi,\rrho,t)=\phi + t + \mu_{\phi}(\phi,\rrho,t),
\\
\label{equ:BFlow_t_R}
\BR(\phi,\rrho,t)=\rrho + \mu_{\rrho}(\phi,\rrho,t),
\end{gather}
where $\mu_{\phi},\mu_{\rrho}:\Hman\times\RRR\to\RRR$ are $\Cinf$ functions flat on $\dHman\times\RRR$.

If, in addition, $\dAp,\dAr:\Hman\to\RRR$ are $\ZZZ$-invariant and $\thdif\in\EBZpl$, then $\tsigma,\,\xi_{\phi},\,\xi_{\rrho}\,\in\,\CiZHR$.
\end{lemma}
\begin{proof}[Proof of~\eqref{equ:coord_func_shift_P}.]
Notice that $\BFld$ defines the following autonomous system of ODE:
$$
\ddd{\phi}{t}=1+\dAp, \qquad \ddd{\rrho}{t}=\dAr.
$$
Let $\thdif=(\thp,\thr)\in\EBpl$ and $z=(\phi,\rrho)\in\Hman$.
If $\rrho$ is sufficiently small, then the shift function for $\thdif$ near $z$ with respect to $\BFld$ can be calculated by the following formula:
\begin{multline}\label{equ:tsigma}
\tsigma(\phi,\rrho)  = 
\int\limits_{\phi}^{\thp(\phi,\rrho)} \frac{d s}{1 + \dAp(s,\rrho)} = 
 \int\limits_{\phi}^{\thp(s,\rrho)} ds + 
\underbrace{ \int\limits_{\phi}^{\thp(\phi,\rrho)} \frac{\dAp(s,\rrho)\, ds }{1 + \dAp(s,\rrho)} }_{\xi_{\phi}(\phi,\rrho)}=\\ =
\thp(\phi,\rrho)-\phi + \xi_{\phi}(\phi,\rrho).
\end{multline}
Evidently, $\xi_{\phi}$ is $\Cinf$.
Moreover, since $\dAp$ is flat on $\dHman$, it follows that so is $\xi_{\phi}$.
This proves~\eqref{equ:coord_func_shift_P}.

{\em Proof of~\eqref{equ:BFlow_t_P} and~\eqref{equ:BFlow_t_R}.}
Applying the previous arguments to $\BFlow_{t}$ and its shift function $t$, we obtain:
$\BP(\phi,\rrho,t)=\phi + t + \mu_{\phi}(\phi,\rrho,t)$, where 
$$ 
\mu_{\phi}(\phi,\rrho,t) = 
\int\limits_{\phi}^{\BP(\phi,\rrho,t)} \frac{\dAp(s,\rrho)\, ds }{1 + \dAp(s,\rrho)}.
$$ 
Moreover, 
$$ 
\mu_{\rrho}(\phi,\rrho,t)=\BR(\phi,\rrho,t)-\rrho =
\int\limits_{0}^{t} \dAr(\BP(\phi,\rrho,s),\BR(\phi,\rrho,s))ds.
$$ 
We claim that $\mu_{\phi}$ and  $\mu_{\rrho}$ are flat on $\dHman\times\RRR$, i.e. at each point $(\phi,0,t)$.

Indeed, since $\dAp$ is flat on $\dHman$, it is easy to see that for any $a,b,c\geq0$
\begin{equation}\label{equ:part_defiv_mu_phi}
\frac{\partial^{a+b+c}\mu_{\phi}}{\partial\phi^a\,\partial\rrho^b\,\partial t^c}
=
\sum_{i} \alpha_{i}(\phi,\rrho,t) + \sum_{j} \int\limits_{\phi}^{\BP(\phi,\rrho,t)} \beta_{j}(s,\rrho)ds,
\end{equation}
\begin{equation}\label{equ:part_defiv_mu_rrho}
\frac{\partial^{a+b+c}\mu_{\rrho}}{\partial\phi^a\,\partial\rrho^b\,\partial t^c}
=
\sum_{k} \gamma_{k}(\phi,\rrho,t) + \sum_{l} \int\limits_{0}^{t} \delta_{l}(s,\rrho)ds,
\end{equation}
where each sum is finite and $\alpha_i$, $\beta_{j}$, $\gamma_{k}$, $\delta_{l}$ are linear combinations with smooth coefficients of partial derivatives of $\dAp$ and $\dAr$ up to order $a+b+c$.
This implies that~\eqref{equ:part_defiv_mu_phi} and~\eqref{equ:part_defiv_mu_rrho} vanish for $\rrho=0$ since $\dAp$ and $\dAr$ are flat on $\dHman$.
Hence $\mu_{\phi}$ and $\mu_{\rrho}$ are flat on $\dHman$.

\medskip 

{\em Proof of~\eqref{equ:coord_func_shift_R}.}
Since $\mu_{\rrho}$ is flat at each point $(\phi,0,\rrho)$, it follows that
$$
\xi_{\rrho}(\phi,\rrho)= \thr(\phi,\rrho) - \rrho=
\BR(\phi,\rrho,\tsigma(\phi,\rrho))-\rrho = 
\mu_{\rrho}(\phi,\rrho,\tsigma(\phi,\rrho))
$$
is flat at $(\phi,0)$.

\medskip

{\em Verification of $\ZZZ$-invariantness.}
Suppose that $\thdif\in\EBZpl$, so
$$
\thp(\phi+2\pi,\rrho)=\thp(\phi,\rrho)+2\pi,
\quad 
\thr(\phi+2\pi,\rrho)=\thr(\phi,\rrho).
$$
Then it is easy to show, see~\cite[Lm.\;6.1]{Maks:Hamvf:2006}, that the functions
$$
\thp-\phi=\tsigma + \xi_{\phi}, \qquad 
\thr-\rrho=\xi_{\rrho}
$$
are $\ZZZ$-invariant.

If in addition $\dAp,\dAr$ are $\ZZZ$-invariant, then $\dAp(s+2\pi,\rrho)=\dAp(s,\rrho)$ and it follows from~\eqref{equ:tsigma}
that
\begin{multline*}
\xi_{\phi}(\phi+2\pi,\rrho)=
\int\limits_{\phi+2\pi}^{\thp(\phi+2\pi,\rrho)} \frac{\dAp(s,\rrho)\, ds }{1 + \dAp(s,\rrho)} = 
\int\limits_{\phi+2\pi}^{\thp(\phi,\rrho)+2\pi} \frac{\dAp(s,\rrho)\, ds }{1 + \dAp(s,\rrho)} =  \\ =
\Bigl| 
\begin{matrix}
\text{substitute} \\ s=s'+2\pi 
\end{matrix}
\Bigr| = 
\int\limits_{\phi}^{\thp(\phi,\rrho)} \frac{\dAp(s'+2\pi,\rrho)\, ds' }{1 + \dAp(s'+2\pi,\rrho)} = 
\int\limits_{\phi}^{\thp(\phi,\rrho)} \frac{\dAp(s',\rrho)\, ds' }{1 + \dAp(s',\rrho)} = \\ =
\xi_{\phi}(\phi,\rrho),
\end{multline*}
whence $\tsigma$ is also $\ZZZ$-invariant.
\end{proof}

\section{Proof of Theorem~\ref{th:charact_period_shift_maps}}\label{sect:proof:th:charact_period_shift_maps}
(a)$\Rightarrow$(b)
Notice that property (b) depends neither on a particular choice of local coordinates at $\orig$ nor on a reparametrization of $\AFld$.
More precisely, let $\hdif:(\disk,\orig)\to(\disk,\orig)$ be a germ of diffeomorphism at $\orig$ and $\beta:\disk\to\RRR\setminus\{0\}$ be everywhere non-zero $\Cinf$ function.
Denote $\BFld = \hdif_{*}(\beta\AFld)$.
Let also $A=J(\hdif,\orig)$ be the Jacobi matrix of $\hdif$ at $\orig$, and let $\lambda_1,\lambda_2$ be the eigen values of the matrix $\nabla\AFld(\orig)$.
Then by~\eqref{equ:nabla_for_pushforward}
 $$\nabla\BFld(\orig) = \beta(\orig) \,\,\cdot\,\, A \cdot \nabla\AFld(\orig) \cdot A^{-1}.$$
Hence the eigen values of $\nabla\BFld(\orig)$ are $\beta(\orig)\lambda_1$ and $\beta(\orig)\lambda_2$.
Therefore $\AFld$ satisfies (b) iff so does $\BFld$.

Suppose now that the period function $\theta:\disk\to(0,+\infty)$ is smooth on all of $\disk$.
Then $\AFlow(z,\theta(z))\equiv z$ for all $z\in\disk$, i.e. $\theta\in\kerA \not=\{0\}$.
Thus $\ShA$ is periodic, whence $1$-jet of $\AFld$ at $\orig$ is non-zero, (see paragraph before Remark~\ref{rem:eigen_val_vanish}).

Consider the following vector field $\BFld=\theta\AFld$.
Let also $\BFlow:\disk\times\RRR\to\disk$ be the flow of $\BFld$.
Then by~\cite{Maks:ReparamShMap} $\BFlow(z,1)\equiv z$ for all $z\in\disk$, i.e. $\BFlow_{1}=\id_{\disk}$.
Hence $\BFlow$ induces the following smooth circle-action $\disk$:
$$
\Gamma: \disk \times S^1 \to \disk,
\qquad \Gamma(z,e^{2\pi i t}) = \BFlow(z,t),
$$
where $S^1$ is regarded as the group $\RRR/\ZZZ$.
Then on the tangent space $T_{\orig}\disk$ we obtain a linear circle action induced by $T_{\orig}\BFlow_t$.
This action is non-trivial since $\jet\BFld=\theta(\orig)\cdot\jet\AFld\not=0$.
Thus we obtain a monomorphism $S^1\to \GLR{2}=\Aut( T_{\orig}\disk )$.

Since every circle subgroup in $\GLR{2}$ is conjugate with $SO(2)$, there exists a linear automorphism $\hdif:\RRR^2\to\RRR^2$ such that 
$$
\hdif\circ \BFlow_t\circ \hdif^{-1}(x,y) =
\left(
 \begin{matrix}
\cos 2\pi t & \sin 2\pi t \\
-\sin 2\pi t & \cos 2\pi t 
\end{matrix}
\right) 
\cdot
\left(
 \begin{matrix} x\\ y \end{matrix}
\right)  + o_t(\sqrt{x^2+y^2}),
$$
whence
$$
\hdif_{*}\BFld(x,y) = -y \frac{\partial}{\partial x} + -x \frac{\partial}{\partial y} + \ \cdots \ \text{terms of high order} \ \cdots.
$$
Notice that $\nabla\hdif_{*}\BFld(\orig) = \left(\begin{smallmatrix} 0 & -1 \\ 1 & 0 \end{smallmatrix}\right)$ and its eigen values are $\pm i$.
As noted above the eigen values of $\nabla\AFld(\orig)$ are $\pm i /\theta(\orig)$.

(b)$\Rightarrow$(c)
In this case $\AFld$ satisfies assumptions of \S\ref{sect:Takens_nf} and therefore can be reduced to one of Takens' normal forms either~\eqref{equ:Takens_flat} or~\eqref{equ:Takens_nonflat}.
But by Corollary~\ref{cor:Takens_nf_polar} in the case \eqref{equ:Takens_nonflat} $\AFld$ has non-closed orbits near $\orig$, which contradicts to the assumption.
Thus by reparametrization and change of coordinated $\AFld$ can be reduced to the form~\eqref{equ:Takens_flat}.

(c)$\Rightarrow$(a)
Let $\eta(\phi,\rrho)=(\phi+2\pi, \rrho)$. 
This map is a lifting of the identity map $\id_{\disk}$.
Since each regular orbit $\orb$ of $\AFld$ is a simple closed curves wrapping around the origin, it follows that $P^{-1}(\orb)$ is diffeomorphic with $\RRR^1$, and in particular it is an orbit of $\BFld$.
Since $P\circ\eta=P$, we obtain $\eta^{-1}P^{-1}(\orb)=P^{-1}(\orb)$, i.e. $\eta$ preserves orbits of $\BFld$.
Hence $\eta\in\EBZpl$.
Then its shift function is given by~\eqref{equ:tsigma}:
$$
\bar\theta(\phi,\rrho)=
\int\limits_{\phi}^{\phi+2\pi} \frac{d s}{1 + \dAp(s,\rrho)} = 2\pi + \txi(\phi,\rrho),
$$
for some $\txi\in\FlatZHR$.
By Lemma~\ref{lm:flat_Zinv_func} $\txi=\xi\circ P$ for some $\xi\in\FlatOR$.
Therefore $\theta= 2\pi + \xi$ is a $\Cinf$ period function for the shift map $\ShA$.

\medskip

Finally, suppose that condition (c) violates.
Then it follows from Lemma~\ref{lm:1jet_TCVF} that ion some local coordinates at $\orig$ the linear part of $\AFld$ is equal to the matrix $A=\left(\begin{smallmatrix}0 & e \\ 0 & 0 \end{smallmatrix}\right)$ for some $e\not=0$.
But the eigen values of $A$ are zero, whence by Remark~\ref{rem:eigen_val_vanish} $\lim\limits_{z\to\orig}\theta(z)=+\infty$.
Theorem~\ref{th:charact_period_shift_maps} is completed.

\section{Proof of Theorem~\ref{th:shmap_periodic}}\label{sect:th:shmap_periodic}
Suppose that the period function $\theta:\disk\to(0,+\infty)$ for $\AFld$ is smooth on all of $\disk$.
Then for each $\afunc\in\Ci{\disk}{\RRR}$ we have that 
\begin{equation}\label{equ:preimage_of _afunc}
\ShA^{-1}\circ\ShA(\afunc) = \{\afunc + n\theta\}_{n\in\ZZZ}.
\end{equation}

(1) We will show that $\imShA=\EApl$, and that the mapping $$\ShA:\Ci{\disk}{\RRR}\to\imShA=\EApl$$ is $\contW{\infty}{\infty}$-open.
This will imply that by~\cite{Maks:LocInv} that $\ShA$ is either a homeomorphism onto $\EApl$ or a $\ZZZ$-covering map.
But due to~\eqref{equ:preimage_of _afunc} $\ShA$ is not injective, whence it is a $\ZZZ$-covering map.

Moreover, by results of~\cite{Maks:LocInv}, for the proof of $\contW{\infty}{\infty}$-openness of $\ShA$ it suffices to construct a local inverse of $\ShA$ defined only on some (arbitrary small) $\Wr{\infty}$-neighbourhood $\Nbh$ of the identity map $\id_{\disk}$ in $\EApl$.

Also notice that by~\cite{Maks:ReparamShMap} the image $\imShA$ of the shift map $\ShA$ and the openness property of $\ShA$ are invariant with respect to the reparametrizations, that is multiplications of $\AFld$ by smooth everywhere non-zero functions.
Hence due to (d) of Theorem~\ref{th:charact_period_shift_maps} we can assume that $\AFld$ is given by the following formula:
\begin{equation}\label{equ:reduce_to_norm_form}
\AFld(x,y) = -(y+\dAy) \dd{x} + (x+\dAx) \dd{y},
\end{equation}
with $\dAx,\dAy\in\FlatOR$.

Then $\AFld$ lifts to the following vector field
$$
\BFld(\phi,\rrho) = (1+\dAp)\,\dd{\phi} + \dAr\,\dd{\rrho},
$$ 
with $\dAp,\dAr\in\FlatHR$.

We need the following statement which will be proved in \S\ref{sect:proof:pr:sect_ShA_all}.
Recall that we denoted by $\EAp{1}$ the subset of $\EApl$ consisting of maps $\hdif$ such that $\jo{\hdif}=\id$.

\begin{proposition}\label{pr:sect_ShA_all}
There exists a $\contW{\infty}{\infty}$-continuous map $$\sigma:\EAp{1} \to \Ci{\disk}{\RRR}$$ which is a section of $\ShA$, i.e.
$$
\hdif(z) = \AFlow(z,\sigma(\hdif)(z)), \qquad \hdif\in\EAp{1}, \ z\in\disk,
$$
satisfies $\sigma(\id_{\disk})=0$, and preserves smoothness.
\end{proposition}

Assuming this statement is proved let us show that $\imShA=\EApl$.

We have that $\imShA\subset\EApl$.
Conversely, let $\hdif\in\EApl$.
Then by Corollary~\ref{cor:j1_of_orb_pres_dif} there exists $\omega\in\RRR$ such that $\hdif_1 = \AFlow_{-\omega}\circ\hdif\in\EAp{1}$.
Then it is easy to see that the function $\afunc = \omega + \sigma(\hdif_1)$ is a shift function for $\hdif$.
Hence $\imShA=\EApl$.

Now by Proposition~\ref{pr:sect_ShA_j1} there exists a $\Wr{1}$ neighbourhood $\Nbh$ of $\hdif$ in $\EApl$ on which the composition 
$$\sigma \circ H:\Nbh\to \Ci{\disk}{\RRR},
\qquad
\sigma \circ H(\hdif) = \omega(\hdif) + \sigma(\AFlow_{-\omega(\hdif)}\circ\hdif)
$$
is a $\contW{\infty}{\infty}$-continuous and preserving smoothness local inverse of $\ShA$.

(3) Since $\Ci{\disk}{\RRR}$ is a contractible Frech\'et manifold and $\ShA$ is a $\ZZZ$-covering map, it follows that $\EApl$ is homotopy equivalent to $S^1$.

(4) Let us show that $\EAd$ is contractible.
Denote by $\FuncVanD$ the subset of $\Ci{\disk}{\RRR}$ consisting of functions vanishing on $\partial\disk$.
\begin{claim}
$\ShA$ yields a $\contW{\infty}{\infty}$-homeomorphism of $\FuncVanD$ onto $\EAd$, whence $\EAd$ is contractible.
\end{claim}
\begin{proof}
Evidently, $\ShA(\FuncVanD)\subset \EAd$.
Conversely, let $\hdif\in\EAd$ and $\afunc$ be any shift function of $\hdif$, so $\hdif(z)=\AFlow(z,\afunc(z))$.
Then $\ShA^{-1}(\hdif)=\{\afunc+ n\theta\}_{n\in\ZZZ}$, where $\theta$ is the period function of $\ShA$.
Recall that $\theta(z)=\Per(z)$ for all $z\not=0$.
Since $\hdif$ is fixed on $\partial\disk$, it follows that $\afunc|_{\partial\disk}=n\theta|_{\partial\disk}$ for a unique $n\in\ZZZ$.
Hence $\afunc'=\afunc-n\theta$ is a unique shift function of $\hdif$ which vanishes on $\partial\disk$, i.e. belongs to $\FuncVanD$.

This implies that $\ShA$ yields a bijection of $\FuncVanD$ onto $\EAd$.
Since in addition $\ShA$ is a local homeomorphism, it follows that it homeomorphically maps $\FuncVanD$ onto $\EAd$.
\end{proof}

(2) Consider the following subset of $\Ci{\disk}{\RRR}$:
$$\Gamma= \{ \afunc\in\Ci{\disk}{\RRR} \ : \ \AFld(\afunc)>-1 \}.$$ 
Then by~\cite{Maks:Shifts} $\Gamma=\ShA^{-1}(\DApl)$, and the restriction $\ShA|_{\Gamma}:\Gamma\to\DApl$ is a $\ZZZ$-covering as well.
Evidently $\Gamma$ is convex, and therefore contractible.
Hence $\DApl$ is homotopy equivalent to $S^1$.

Moreover, $\ShA$ homeomorphically maps the following convex set $\Gamma\cap\FuncVanD$ onto $\EAd$.
Whence $\EAd$ is contractible as well.

This completes Theorem~\ref{th:shmap_periodic} modulo Proposition~\ref{pr:sect_ShA_all}.

\section{Proof of Proposition~\ref{pr:sect_ShA_all}}\label{sect:proof:pr:sect_ShA_all}
Let $\hdif\in\EAp{1}$.
First we will explain how to construct $\sigma$ and then show its continuity.

Notice that $\EAp{1} \subset \MapR$.
Then by Proposition~\ref{pr:lift_Map_id} there exists a lifting $\thdif=\ml(\hdif)\in\MapZH$.
Let $\thdif=(\thp,\thr)$ be the coordinate functions of $\thdif$.

\begin{claim}
$\ml(\EApl) \subset \EBZpl$.
In particular, $\thdif\in \EBZpl$.
\end{claim}
\begin{proof}
Let $\torb$ be the orbit of $\BFld$.
We have to show that $\thdif(\torb)\subset\torb$.

Notice that $\orb=P(\torb)$ is the orbit of $\AFld$, and $\torb=P^{-1}(\orb)$, that is the inverse image of $\orb$ is connected.
Moreover, $\hdif(\orb)=\orb$, whence 
$$ P \circ \thdif(\torb) = \hdif\circ P(\torb) = \hdif(\orb)=\orb.$$
Therefore $\thdif(\torb) \subset P^{-1}(\orb) =\torb.$
\end{proof}

Now by Lemma~\ref{lm:shift-maps-without-sing-R2}  there exists a unique $\Cinf$ shift function $\tsigma$ of $\thdif$ with respect to $\BFld$.
Moreover, by Lemma~\ref{lm:sect_ShB} $\tsigma$ is $\ZZZ$-invariant.
We have to show that $\tsigma=\sigma\circ P$ for some $\Cinf$ function $\sigma:\disk\to\RRR$.
Then $\sigma$ will be a shift function for $\hdif$.

Again by Lemma~\ref{lm:sect_ShB} 
$$
\tsigma = \thp(\phi,\rrho)-\phi + \widetilde{\xi}(\phi,\rrho),
$$
where $\widetilde{\xi}\in\FlatZHR$, i.e. it is $\ZZZ$-invariant and flat on $\dHman$.
Then by Lemma~\ref{lm:flat_Zinv_func} there exists a $\Cinf$ function $\xi\in\FlatOR$ such that $\widetilde{\xi}=\xi\circ P$.
Therefore we have to show that there exists a $\Cinf$ function $\nu\in\Ci{\disk}{\RRR}$ such that 
\begin{equation}\label{equ:Delta_Phi}
\thp(\phi,\rrho)-\phi=\gamma\circ P(\phi,\rrho).
\end{equation}

\begin{lemma}\label{lm:h_zg}
Regard $\hdif$ as a smooth function $\hdif:\disk\to\CCC$.
Then there exists a smooth function $\gamma:\disk\to\CCC$ satisfying~\eqref{equ:Delta_Phi} and such that $\hdif(z)=z\gamma(z)$ and $\gamma(0)=1$.

Moreover, the correspondence $\hdif\mapsto\gamma$ is a $\contS{\infty}{\infty}$-continuous and preserving smoothness map $\EApl\to\Ci{\disk}{\CCC}$.
\end{lemma}
\begin{proof}
Notice that $\gamma(z)=\hdif(z)/z$ is a smooth map $\gamma:\disko\to\CCC\setminus\{\orig\}$.
We have to prove that it smoothly extends to a map $\disk\to\CCC$.

We will only show that
\begin{equation}\label{equ:hh_zz_mu}
\hdif\bar\hdif=z \bar{z} + \zeta(z),
\end{equation}
where $\zeta:\disk\to\CCC$ is a smooth map being flat at $\orig$.
Then smoothness of $\gamma$ at $\orig$ will follow by the arguments similar to the proof of~\cite[Lm.~31]{Maks:Shifts}.

Let $z=re^{i\phi}=P(\phi,\rrho)$ and $\thdif=(\thp,\thr)=\ml(\hdif)$.
Then 
$$
\hdif\bar\hdif(z)   =
\thr^2(\phi,\rrho)  \stackrel{\eqref{equ:coord_func_shift_R}}{=\!=\!=\!=}
(r + \xi_{\rrho}(\phi,\rrho))^2 = 
r^2 + \tzeta(\phi,\rrho)=
z \bar{z} + \tzeta(\phi,\rrho),
$$
for some $\tzeta\in\FlatZHR$.
By Lemma~\ref{lm:flat_Zinv_func} $\tzeta=\zeta\circ P$ for some $\zeta\in\FlatOR$, whence 
$\hdif\bar\hdif(z)=z \bar{z} + \zeta(z)$.
This establishes~\eqref{equ:hh_zz_mu} and smoothness of $\gamma$.
Continuity of the division $\hdif\mapsto\gamma$ is implied by compactness of $\disk$ and the following lemma being a very particular case of results of~\cite{MostowShnider:TrAMS:1985}:
\begin{lemma}\label{lm:division_by_z}{\rm\cite{MostowShnider:TrAMS:1985}.}
Let $Z:\Ci{\CCC}{\CCC}\to \Ci{\CCC}{\CCC}$ be the ``multiplication by $z$'' map, i.e.
$$
Z(\gamma)(z) = z \, \gamma(z), \qquad \gamma\in \Ci{\CCC}{\CCC},\  z\in\CCC.
$$
Denote by $\im Z$ the image of $Z$ in $\Ci{\CCC}{\CCC}$.
Then $Z$ is injective and the inverse map $Z^{-1}:\im Z \to \Ci{\CCC}{\CCC}$ is $\contS{\infty}{\infty}$ continuous.
\end{lemma}

Let us prove~\eqref{equ:Delta_Phi}.
Since $\gamma\not=0$ on $\disk$ and $\disk$ is simply connected, there exists a unique smooth function $\Gamma:\disk\to\RRR$ such that $\Gamma(\orig)=0$, and 
$\gamma(z) = |\gamma(z)| e^{i\Gamma(z)}$.

Let $z=re^{i\phi}\in\disk$.
Then $\hdif\circ P(\phi,\rrho) = P \circ \thdif(\phi,\rrho)$ implies that 
$$
\thr e^{i \thp(\phi,\rrho)} = \hdif(\rrho e^{i\phi})=
\gamma(\rrho e^{i\phi})\, \rrho e^{i\phi}=
|\gamma(z)| \rrho e^{i(\Gamma(\rrho e^{i\phi})+ \phi)}.
$$
Hence $$\thp(\phi,\rrho)-\phi=\Gamma(r e^{i\phi})+2\pi n$$ for some $n\in\ZZZ$.
In order to find $n$ put $\rrho=0$, then
$$2\pi n = \thp(\phi,0)-\phi - \Gamma(\orig) = \phi-\phi+0=0,$$
whence $n=0$.
\end{proof}

Hence $\sigma=\gamma+\xi$ is a $\Cinf$ shift function for $\hdif$.

It remains to note that the correspondences $\hdif\mapsto\gamma$ and $\hdif\mapsto\xi$ are $\contW{\infty}{\infty}$-continuous and preserving smoothness by Lemmas~\ref{lm:flat_Zinv_func} and~\ref{lm:h_zg}.
Theorem~\ref{th:shmap_periodic} is completed.

\section{Proof of Theorem~\ref{th:func_flat_pert}}\label{sect:proof:th:func_flat_pert}
We have to find a $\Cinf$ function $\gfunc:\RRR\to\RRR$ such that $$j^{\infty}\func(x,y)=j^{\infty}\gfunc(x^2+y^2),$$
where $j^{\infty}$ means \myemph{$\infty$-jet at $\orig$}.

Recall that a vector field $\AFld$ is  \myemph{parameter rigid} if for any other vector field $\AFld'$ such that every orbit $\orb'$ of $\AFld'$ is included in some orbit of $\AFld$, there exists a $\Cinf$ function $\nu:\disk\to\RRR$ such that $\AFld'=\nu\AFld$, see~\cite{Santos:ETDS:2007, Maks:jets}.

\begin{claim}\label{clm:F_is_param_rig}
 $\AFld$ is parameter rigid.
\end{claim}
\begin{proof}
We have that $\imShA=\EApl$ and that local sections of $\ShB$ preserve smoothness.
These two assumptions imply parameter rigidity of $\AFld$ by results of~\cite{Maks:jets}.
\end{proof}

Choose local coordinates at $\orig$ in which $\AFld$ is given by~\eqref{equ:Takens_nf1_xy}.
Let also $\AFld'(x,y) = -\func'_{y} \dd{x} + \func'_{x} \dd{y}$ be the Hamiltonian vector field of $\func$.
Then by parameter rigidity of $\AFld$ there exists a $\Cinf$ function $\nu:\disk\to\RRR$ such that 
\begin{equation}\label{equ:param_rig_implication}
\func'_{x}(x,y)= (x+\dAx)\,\nu(x,y),
\qquad 
\func'_{y}(x,y)= (y+\dAy)\,\nu(x,y).
\end{equation}
Hence 
\begin{equation}
x \cdot j^{\infty}\func'_{y} = y \cdot  j^{\infty}\func'_{x}.
\end{equation}

\begin{lemma}\label{lm:xpy_ypx}
Let $p$ be a non-zero homogeneous polynomial of degree $n$ in two variables satisfying the following identity:
\begin{equation}\label{equ:xpy_ypx}
x p'_{y} = y p'_{x}.
\end{equation}
Then $n=2k$ is even and $p(x,y)=a(x^2+y^2)^k$ for some $a\in\RRR$.
\end{lemma}
\begin{proof}
Since $p$ is homogeneous of degree $n$, it satisfies the following Euler identity:
$np = xp'_{x} + y p'_{y}.$
Hence 
$$
nxp = x^2 p'_{x} + xy p'_{y} =x^2 p'_{x} + y^2 p'_{x} =  (x^2+y^2)\,p'_{x}.
$$
Therefore either $n=0$ and $p$ is a constant, or $p(x,y) =(x^2+y^2) q(x,y)$ for some homogeneous polynomial of degree $n-1$.
We claim that $q$ satisfies $x q'_{y} = y q'_{x}$.
Indeed,
$$
p'_{y} = 2yq+(x^2+y^2)q'_{y},
\qquad
p'_{x} = 2xq+(x^2+y^2)q'_{x}.
$$
Then we get from~\eqref{equ:xpy_ypx} that $(x^2+y^2)xq'_{y}=(x^2+y^2)y q'_{x}$, and therefore $x q'_{y} = y q'_{x}$.

By the same arguments, $q$ is divided by $x^2+y^2$ as well.
Now lemma follows by induction on the degree $n$.
\end{proof}

Let $j^{\infty}\func(x,y)= \sum\limits_{i=0}^{\infty} p_i(x,y)$ be the Taylor series of $\func$, where $p_i(x,y)$ is a homogeneous polynomial of degree $i$.
Then every $p_i$ satisfies~\eqref{equ:xpy_ypx} whence $j^{\infty}\func(x,y)= \sum\limits_{i=0}^{\infty} a_i(x^2+y^2)^i$ for some $a_i\in\RRR$.
By a theorem of E.~Borel there exists a $\Cinf$ function $\gfunc:\RRR\to\RRR$ such that 
$$
j^{\infty}\gfunc(t) = \sum_{i=0}^{\infty} a_i t^i, 
$$
whence 
$j^{\infty}\func(x,y)=j^{\infty}\gfunc(x^2+y^2)$, which completes our statement.

\section{Acknowledgenemts}
The author is grateful to Ye.~Polulyakh, V.~Krouglov, and D.~Ilyutko for useful discussions of Proposition~\ref{pr:collin-conjug-lin-maps}.

\end{document}